\documentclass[a4paper,11pt]{article}
\usepackage{amsmath,amsthm,amssymb,enumerate}
\usepackage{graphicx}
\usepackage{euscript,mathrsfs}
\usepackage{dsfont}
\usepackage{xcolor}
\usepackage{empheq}
\usepackage{bm}
\usepackage{geometry}
\usepackage{algorithm}
\usepackage{algorithmicx}
\usepackage{algpseudocode}
\usepackage{algcompatible}
\usepackage{color}
\catcode`\@=11 \@addtoreset{equation}{section}

\catcode`\@=12

\usepackage{upgreek}
\usepackage{cancel}
\usepackage{soul}
\usepackage[normalem]{ulem}

\usepackage{tcolorbox}
\tcbuselibrary{skins}

\usepackage{stmaryrd}
\allowdisplaybreaks

\usepackage{braket,amsfonts}
\usepackage[colorlinks=true,linkcolor=blue,filecolor=magenta,urlcolor=cyan]{hyperref}

\usepackage{tikz}
\usetikzlibrary{patterns}
\usepackage{booktabs}
\usepackage{tikz,pgfplots,pgfplotstable}
\pgfplotsset{compat=1.15}
\usetikzlibrary{matrix,arrows.meta,calc,shapes}
\tikzset{
  centered/.style = { align=center, anchor=center },
     empty/.style = { font=\sffamily\Large, centered, text width=2cm },
       box/.style = { font=\sffamily, fill=green, centered },
    result/.style = { font=\sffamily\scriptsize, fill=black!20, centered},
     arrow/.style = { very thick, color=red, ->, >=Triangle},
}

\usepackage{listings} 

\usepackage{scalerel,stackengine}
\stackMath
\newcommand\reallywidehat[1]{%
\savestack{\tmpbox}{\stretchto{%
  \scaleto{%
    \scalerel*[\widthof{\ensuremath{#1}}]{\kern-.6pt\bigwedge\kern-.6pt}%
    {\rule[-\textheight/2]{1ex}{\textheight}}
  }{\textheight}%
}{0.5ex}}%
\stackon[1pt]{#1}{\tmpbox}%
}

\newtheorem{Theorem}{Theorem}[section]

\newtheorem{Lemma}[Theorem]{Lemma}
\newtheorem{Corollary}[Theorem]{Corollary}

\newtheorem{Definition}[Theorem]{Definition}

\newtheorem{Remark}[Theorem]{Remark}

\newcommand{\M}{\mathbb{M}}

\newcommand{\Mhk}{{\M}_h^k}
\newcommand{\Mk}{{\M}^k}
\newcommand{\bfw}{\mathbf{w}}

\newcommand{\bfe}{\mathbf{e}}

\newcommand{\detJ}{J}
\newcommand{\Jacob}{{\mathbb F}}

\newcommand{\xx}{x_1}
\newcommand{\xy}{x_2}

\newcommand{\xref}{\widehat{x}}
\newcommand{\xrefx}{\xref_1}
\newcommand{\xrefy}{\xref_2}

\newcommand{\eu}{e_\vu}
\newcommand{\euk}{e_\vu^k}

\newcommand{\rmS}{{\rm S}}
\newcommand{\rmT}{{\rm T}}

\newcommand{\bfphi}{\boldsymbol{\varphi}}
\newcommand{\hbfphi}{\widehat{\bfphi}}

\newcommand{\hq}{\widehat{q}}

\newcommand{\bftau}{\mathbb{T}}
\newcommand{\hbftau}{\widehat{\bftau}}

\newcommand{\summ}{\sum_{k=1}^m}

\newcommand{\sumN}{\sum_{k=1}^N}

\newcommand{\Div}{{\rm div}\,}
\newcommand{\Grad}{\nabla}

\newcommand{\Gradref}{\Grad_{\xref}}

\newcommand{\Laph}{ \pd_{\xx,h}^2 }
\newcommand{\pdx}{ \pd_{\xx}}
\newcommand{\Lapx}{ \pdx^2}

\newcommand{\dvolref}{\,{\rm d}\xref}

\newcommand{\ds}{\,{\rm d}\xx}

\newcommand{\Abs}[1]{ \left| #1 \right|}
\newcommand{\abs}[1]{ | #1 |}

\newcommand{\norm}[1]{\left\lVert#1\right\rVert}

\newcommand{\aleq}{\stackrel{<}{\sim}}

\newcommand{\vr}{\varrho}
\newcommand{\vrf}{\vr_f}
\newcommand{\vrs}{\vr_s}
\newcommand{\vu}{\mathbf{u}}

\newcommand{\txi}{\widetilde{\xi}}

\newcommand{\vw}{\bfw}
\newcommand{\ve}{\bfe}
\newcommand{\er}{{\ve}_2}

\newcommand{\hvu}{\widehat{\vu}}
\newcommand{\hvw}{\widehat{\vw}}
\newcommand{\hvv}{\widehat{\vv}}

\newcommand{\hp}{\widehat{p}}

\newcommand{\FSIref}{\widehat{W}}

\newcommand{\hFSIh}{\widehat{W}_{\eta_h}}
\newcommand{\vc}[1]{{\bf #1}}

\newcommand{\vv}{\vc{v}}

\newcommand{\I}{\mathbb{I}}
\newcommand{\Id}{\I}

\newcommand{\R}{\mathbb{R}}

\definecolor{Cgrey}{rgb}{0.85,0.85,0.85}
\definecolor{Cblue}{rgb}{0.50,0.85,0.85}
\definecolor{Cred}{rgb}{1,0,0}
\definecolor{fancy}{rgb}{0.10,0.85,0.10}
\definecolor{forestgreen}{rgb}{0.13, 0.55, 0.13}

\newcommand{\vx}{{\bm x}}
\newcommand{\vxref}{\widehat{\bm x}}

\newcommand{\TS}{\tau}
\newcommand{\pd}{\partial}
\newcommand{\pdt}{\pd _t}
\newcommand{\pdtt}{\pd _t^2}

\newcommand{\PDt}{D_t}

\newcommand{\Qfh}{Q^f_h}

\newcommand{\Vsh}{V^s_h}
\newcommand{\Vshz}{V^s_{0,h}}

\newcommand{\hQfh}{\widehat{Q}^f_h}
\newcommand{\hVfh}{\widehat{V}^f_h}

\newcommand{\ALE}{\mathcal{A}_{\eta}}

\newcommand{\intS}[1] {\int_\Sigma #1 \ds }

\newcommand{\intSB}[1] {\int_\Sigma \left( #1 \right) \ds }

\newcommand{\Of}{\Omega_\eta}

\newcommand{\Ofh}{\Omega_{\eta_h}}

\newcommand{\intOref}[1]{\int_{\Oref} #1 \dvolref}
\newcommand{\intOrefB}[1]{\int_{\Oref}\left( #1 \right)\dvolref}

\newcommand{\Oref}{{\widehat{\Omega}}}
\newcommand{\gridf}{\mathcal{T}_h}
\newcommand{\grids}{ \Sigma_h}
\newcommand{\calP}{ \mathcal{P}}
\newcommand{\calL}{ \mathcal{L}}

\newcommand{\PiF}{{\Pi_h^f}}
\newcommand{\Piq}{ \Pi_h^Q}
\newcommand{\Riesz}{{\mathcal{R}_h^s}}

\begin{document}


\title{Stability and error estimates of a linear and partitioned finite element method approximating nonlinear fluid--structure interactions
}

\author{
\and Bangwei She\thanks{Academy for Multidisciplinary Studies, Capital Normal University
(bangweishe@cnu.edu.cn).
} 
\and Tian Tian\thanks{School of Mathematics, Statistics and Mechanics, Beijing University of Technology (tiantian@bjut.edu.cn)}$^{\dagger}$
\and Karel T\r{u}ma\thanks{Mathematical Institute, Faculty of Mathematics and Physics, Charles University (ktuma@karlin.mff.cuni.cz)}$^{\dagger}$
}
\date{}
\maketitle
\vspace*{-8mm}

\begin{abstract}

We propose and analyze a linear and partitioned finite element method for fluid–shell interactions under the arbitrary Lagrangian-Eulerian (ALE) framework. We adopt the P1-bubble/P1/P1 elements for the fluid velocity, pressure, and structure velocity, respectively. We show the stability and error estimates of the scheme without assuming infinitesimal structural deformation nor neglecting fluid convection effects. The theoretical convergence rate is further corroborated by numerical experiments.

\medskip


\end{abstract}



\section{Introduction}\label{sec:1}

The computational modeling of fluid-structure interactions (FSI) has garnered extensive attention over recent decades due to its significance in biological systems, aerospace engineering, and energy technologies \cite{Bazilevs,Bodnar,Tezduyar}. A large number of numerical methods have been proposed, ranging from monolithic schemes to modular partitioned strategies. 
The monolithic methods solve the coupled large fluid and structure system, see e.g. Gee et al. \cite{Gee}, Hecht and O. Pironneau \cite{Pironneau}, and Lozovskiy et al. \cite{Lozovskiy,Lozovskiy2}. This type of methods are usually more stable due to the strong coupling between fluid and structure. 
The \emph{partitioned} (sometimes called {splitting or kinematically coupled or loosely coupled}) methods are usually cheaper as the large FSI system is divided into two smaller subsystems, see e.g. Badia et al. \cite{Badia}, Buka\v{c} et al. \cite{Boris1},  Hundertmark and Luk\'{a}\v{c}ov\'{a}-Medvid'ov\'{a} \cite{Hundertmark}, 
Luk\'{a}\v{c}ov\'{a}-Medvid'ov\'{a} et al. \cite{Lukacova}. 
To maintain the stability, suitable boundary conditions are needed at the interaction interface, see the comparison of different types of partitioned methods in \cite{Fernandez13}.

Despite the practical success of partitioned schemes, their rigorous convergence analysis remains a highly delicate and technically demanding task in general. Significant progress in this direction has been achieved in a series of influential works, including those by Buka\v{c} and Muha \cite{Boris2}, Burman et al. \cite{Burman,Burman2}, Fern'{a}ndez and Mullaert \cite{Fernandez}, Li et al. \cite{Buyang24}, Seboldt and Buk\v{c} \cite{Seboldt}, among others. These contributions provide deep insights into the stability and convergence mechanisms of partitioned methods with careful treatment of the interaction interface. 

Owing to the intrinsic nonlinearity of the coupled fluid–structure system, existing analytical results are typically established within a simplified yet mathematically tractable framework, often assuming infinitesimal solid displacements and neglecting convective effects in the fluid equations. The objective of the present paper is to extend the theoretical understanding of partitioned schemes beyond this commonly adopted setting. More precisely, we establish the energy stability of a linear, partitioned finite element method (see Theorem~\ref{Thm_Sta}) together with its linear convergence rate (see Theorem~\ref{theorem_conv_rate}), without invoking the above-mentioned assumptions.


The plan of the paper is the following. In Section~2 we describe the problem and present our numerical method. In Section~3 we prove the energy stability of the numerical solution. In Section~4 we analyze the convergence rate of the numerical solution toward a strong solution. In Section~5 we present the numerical experiments. Section~6 is the conclusion.

\section{Problem description and numerical method}\label{Sec_wf}
In this section, we introduce the notations, function spaces, and some preliminary estimates.
\subsection{Problem formulation}
  In this paper, we are interested in the interaction between an incompressible viscous fluid and a thin deformable shell structure. Specifically, we consider the structure to be the upper boundary $\Gamma_S$ of the fluid domain. It results in a nonlinear coupled system with a time dependent fluid domain given by 
\[ \Of(t)=\left\{\vx =(\xx,\xy) \in \Sigma \times (0, \eta(t,\xx)) \right\} \subset \R^2, \; \Sigma = (0,L), \mbox{ and } \Gamma_S =\left\{\vx =(\xx, \eta(t,\xx)) \mid \xx \in \Sigma \right\},
\]
where $\eta=\eta(t,\xx)$ represents the height of the upper boundary $\Gamma_S(t)$. 
Typically, the fluid dynamics are described by the balance equations of mass and momentum in the current (Eulerian) configuration, while the structure dynamics are formulated in the reference (Lagrangian) configuration. 
The motion of the coupled FSI system is governed by the following equations. 
\begin{equation}\label{pde_f}
\left\{
\begin{aligned}
&\Div \vu =0 , &\quad \text{ in } (0,T) \times \Of, 
\\
&
\vrf \left( \pdt \vu + (\vu \cdot \Grad) \vu \right) - \Div \bftau(\vu, p) =0,  &\quad \text{ in } (0,T) \times \Of,
\\
& 
\vrs \pdtt \eta  + \calL (\eta) = f, \quad \xi=\pdt \eta,  &\quad \text{ on } (0,T) \times  \Sigma. 
\end{aligned}
\right.
\end{equation}
Here,  
    $\bftau=  2\mu (\Grad \vu)^\rmS -p\I,  (\Grad \vu)^\rmS = (\Grad \vu+ (\Grad \vu)^\rmT)/2, 
    \calL (\eta) = - \gamma_1  \Lapx \eta - \gamma_2 \Lapx \zeta,   \zeta = - \Lapx \eta, $
 $\vu = \vu(t, \vx)$ and $p = p(t, \vx)$ are the fluid velocity and pressure, respectively. 
The symbols $\vrf$ and $\vrs$ denote the densities of the fluid and the structure, respectively. 
The constants $\gamma_1 > 0$ and $\gamma_2 > 0$ are given parameters. 
$\xi = \xi(t, \vx)$ represents the structure velocity, 
and $f = f(t, \vx)$ denotes the interaction force acting on the structure due to the Cauchy stress of the fluid.
The problem is closed by the initial data 
\begin{equation}\label{pde_ini}
\vu(0) = \vu_0 \text{ in } \Of(0), \quad 
\eta(0,\cdot) = \eta_0>0, \; \xi(0,\cdot) = \xi_0 \; \text{in } \Sigma,
\end{equation}
together with no-slip boundary conditions 
\[\vu|_{\Gamma_D} =0, \ \Gamma_D = \pd \Of \setminus \Gamma_S, \mbox{ and } \pdx \eta|_{\pd \Sigma}=0\]
 and the coupling conditions on the interface:
\begin{equation}\label{pde_bdc}
   \vu|_{\Gamma_S} = \xi \er, \quad
    f = -\,\er \cdot \big( \detJ \bftau(\vu,p) \circ \ALE \, \Jacob^{-T} \big) \cdot \er,
\end{equation}
where $\ALE$ is an invertible mapping from the reference domain $\Oref$ to the time-dependent fluid domain $\Of$ that is introduced to handle the inconsistency between the Eulerian description of the fluid and the Lagrangian description of the structure. This allows us to describe the entire problem in the single common configuration $\Oref$. 
Without loss of generality, we set $\Oref = \Sigma \times (0,1)$. 
In this work, the mapping $\ALE$ is explicitly constructed as a linear extension:
\begin{subequations}\label{Jacob}
\begin{equation}
\ALE: \Oref \mapsto \Of, \quad 
(\xrefx, \xrefy) \mapsto (\xx,\xy) = \ALE(t,\vxref) 
= \left( \xrefx , \eta \, \xrefy \right),
\end{equation}
with the corresponding Jacobian and determinant given by
\begin{equation}
\Jacob = \Gradref \ALE =
\begin{pmatrix}
1 & 0 \\[3pt]
\xrefy \, \frac{\partial\eta}{\partial\hat{x}_1} & \eta
\end{pmatrix},
\qquad
\detJ = \det(\Jacob) = \eta.
\end{equation}
\end{subequations}
Here and in what follows, for any generic function $v$ defined on the current domain, 
its pullback to the reference domain is denoted by $\widehat{v} = v \circ \ALE$.


\paragraph{Function spaces}We use the standard notations of $W^{k,p}(D)$, $W^{k,p}_0$, and $L^p(D)$ as Sobolev space, Sobolev space with vanishing trace, and Lebesgue space on a generic domain $D$, respectively. Moreover, we introduce the following coupled space to recover the no-slip conditions on the interface.
\begin{align*}
\FSIref = \left\{(\hbfphi,\psi)\in W^{1,2}(\Oref) \times W^{1,2}(\Sigma) \mid \hbfphi=0\text{ on }{\Gamma}_D,\; \psi(\xrefx) \er =\hbfphi(\xrefx,1) \right\}.
\end{align*}
With the above notation and via the change of coordinates, we introduce the weak formulation of the FSI problem \eqref{pde_f}--\eqref{pde_bdc} on the reference domain $\Oref$, see e.g. \cite{SST}. 
\begin{Definition}[Weak formulation on the reference domain $\Oref$]\label{Lwfref} 
We say the following formula is a weak formulation of the FSI problem on $\Oref$.
\begin{subequations}\label{wf_ref}
\begin{equation}\label{wf1_ref}
\intOref{    \hq  \Grad \hvu :  \M} =0 \quad  \mbox{ for all } q \in L^2_0(\Oref);
\end{equation}
\begin{equation}\label{wf2_ref}
\begin{aligned}
&\vrf \intOref{ \left( \detJ \pdt \hvu  +   \pdt \detJ  \frac{\hvu}2 \right)\cdot \hbfphi } 
+ \frac12 \vrf \intOref{   \big( \hbfphi \cdot (\Grad \hvu)   -  \hvu \cdot (\Grad \hbfphi)   \big) \cdot  \Jacob^{-1} \cdot \hvv \detJ }
\\&
+ \intOref{   \hbftau(\hvu,\hp) : \big(\Grad \hbfphi \Jacob^{-1} \big)\detJ } 
+ \vrs\intS{\pdt \xi  \psi} + a_s(\eta, \zeta,\psi)=0,
\end{aligned}
\end{equation}
\end{subequations}
for all $(\hbfphi,\psi)\in\FSIref$, where $\hvv = \hvu -\hvw$, $\hvw =\pdt \ALE$, $\zeta=-\Lapx \eta$,  $\ALE$ and $\Jacob$ are given in \eqref{Jacob}, and 
\begin{equation}\label{tauref}
\begin{aligned}
& a_s(\eta, \zeta, \psi) =   \intS{ (\gamma_1 \pdx  \eta + \gamma_2 \pdx \zeta)   \pdx \psi },
\\&
\hbftau(\hvu,\hp) = \bftau(\vu,p) \circ \ALE 
=  2 \mu ( \Grad \hvu \Jacob^{-1} )^\rmS -  \hp \Id,  
\quad \M =\detJ \Jacob^{-T} . 
\end{aligned}
\end{equation}
\end{Definition}

\paragraph{Finite element space}
Let $\grids$ be the surface mesh of $\gridf$ on the top boundary $\widehat{\Gamma}_S = {\Gamma}_S \circ \ALE$ and 
let $\gridf$ be a shape regular and quasi-uniform triangulation of the reference domain $\Oref$, where $h$ stands for the maximum diameter of all elements of $\gridf$.  We denote by $K \in \gridf$ a generic element in $\gridf$ and by $\sigma \in \grids$ a generic face element in $\grids$. 
Moreover, we introduce the following function spaces on $\Oref$ 
\begin{align*}
& \hVfh = \left\{\hbfphi \in W^{1,2}(\Oref;\R^d) \middle|  \hbfphi \in \calP^1(K)   \oplus B_1(K)  , \; \forall K \in \gridf, \hbfphi|_{\Gamma_D}=\mathbf{0}, \widehat{\varphi}_1 
|_{\widehat{\Gamma}_S}=0
\right\},
\quad
\\& B_1(K)=\left\{\phi \in \calP^3(K) \middle|  \phi(a_i)=0, \quad \text{where } a_i, \;  i=1,2,3, \ \text{are\ vertices\ of}\ K \in \gridf \right\}, 
\\&
\hQfh = \left\{\hq \in {C^0} (\Oref) \middle|  \hq \in \calP^1(K), \; \forall K \in \gridf \right\}, 
\quad 
\Vsh= \left\{ \psi \in W^{1,2}_0(\Sigma)\middle| \psi \in \calP^1(\sigma), \; \forall \sigma \in \grids \right\}  ,
\end{align*}
where $\calP^n(K)$ (resp. $\calP^n(\sigma)$) denote polynomials of degree not greater than $n$ on $K$ (resp. on $\sigma$).   
In order to deal with the fourth order term we introduce a discrete Laplace  
 $\zeta_h = - \Laph \eta_h \in \Vshz := \Vsh \cap L^2_0(\Sigma)$ that is uniquely defined by  
 \begin{equation}\label{laph}
 \intS{ \zeta_h\; \psi }  +\intS{ \pdx \eta_h \pdx \psi } =0 \quad \mbox{for all } \psi \in \Vsh.
 \end{equation}
With the above notation, we have for all $\psi \in \Vsh$ that
\begin{equation}\label{as2}
    a_s(\eta_h, \zeta_h, \psi) =   \intS{ (\gamma_1 \pdx  \eta_h  + \gamma_2 \pdx \zeta_h)   \pdx \psi } 
    = \intS{ (\gamma_1  \zeta_h  - \gamma_2 \Laph \zeta_h )  \psi }.
 \end{equation}
Next, recalling Boffi et al.~\cite{boffi} and Schwarzacher et al. \cite{SST} we know there exist an interpolation operator $\Piq: \;  L^2(\Oref)   \; \mapsto \;   \Qfh$ and a Riesz projection operator $\Riesz: W^{1,2}(\Sigma) \mapsto \Vsh$, such that 
\begin{equation*}
\begin{aligned}
\norm{\Piq p - p }_{W^{k,s}} \aleq h \norm{p}_{W^{k+1,s}}, \; k=1,2, \; s \in [1,\infty], 
\end{aligned}
\end{equation*}
\begin{align*}
\intS{\pdx (\Riesz \eta - \eta)  \pdx \psi} =0  \quad \forall \; \psi \in \Vsh \quad \text{ with } \intS{ \Riesz \eta}=\intS{ \eta} \mbox{ for all } \eta\in W^{1,2}(\Sigma),
\end{align*}
\begin{align*}
\intS{(\Laph \Riesz \eta - \Lapx \eta)  \; \psi} = 0 \quad \forall \; \psi \in \Vsh \mbox{ and } \eta \in W^{2,2}(\Sigma),
\end{align*}
\begin{equation}\label{RieszE}
\norm{\pdx \Riesz \eta}_{L^2(\Sigma)} \aleq \norm{\pdx \eta}_{L^2(\Sigma)},
\, 
\norm{\Laph \Riesz \eta}_{L^2(\Sigma)} \aleq \norm{\Lapx \eta}_{L^2(\Sigma)}, 
\,
\norm{\Laph \Riesz \eta - \Lapx \eta}_{L^2(\Sigma)} \aleq h \norm{\eta}_{ W^{3,2}(\Sigma)},
\end{equation}
where by $a\aleq b$ we mean $a\leq c b$ for a positive constant $c$ that is independent of the computational parameters $\TS$ and $h$. 

Further, we recall an interpolation operator $\PiF: W^{k,p}(\Oref) \mapsto \hVfh$ for the fluid velocity from \cite{SST} that matches the Dirichlet boundary condition.   
\begin{Theorem}(\cite[Theorem 5.6]{SST})
\label{thm:projection-velocity}
Let $\Ofh\subset \R^2$ be a subgraph, the grids $\mathcal{T}_h$ and $\Sigma_h$ respectively defined on the domains $\widehat{\Omega}$ and $\Sigma$ be shape regular and quasi-uniform. Besides, let $\Gamma_S=\{(x_1,\eta_h(x_1))| x_1\in\Sigma\}$ satisfying $ \min_\Sigma \eta_h \geq \delta $ and $\norm{\partial_{\xx} \eta_h}_{L^\infty} \leq L$. Assume moreover that $ \min_\Sigma \eta \geq \delta $ and $\norm{\partial_{\xx} \eta}_{L^\infty} \leq L$.
Then there exists 
\begin{align*}
\PiF: \;  \; \left\{  (\hvu,\xi) \in \FSIref \mid \Div\vu=0\text{ on }\Omega_\eta \right\}\to \; 
\hFSIh
:=\left\{(\hbfphi,\psi)\in  \hVfh \times \Vsh \middle| \hbfphi(\xrefx,1)=\psi(\xrefx)\er\right\}
\end{align*}
 satisfying for $\gamma <\infty$ that  
\begin{align}
\label{iu}
\begin{aligned}
 \norm{ \hvu- \PiF \hvu}_{L^\gamma(\Oref)} + h  \norm{\Grad ( \hvu- \PiF \hvu)}_{L^2(\Oref)} &\aleq  h^2 \norm{\hvu}_{W^{2,2}(\Oref)}+ h^2 \norm{\xi}_{W^{2,2}(\Sigma)}+ h\norm{\eta-\eta_h}_{W^{1,2}(\Sigma)}, 
\\
\norm{\PiF \hvu}_{L^\gamma(\Oref)} +  \norm{\Grad  \PiF \hvu}_{L^2(\Oref)} &\aleq\norm{\hvu}_{W^{1,2}(\Oref)}+\norm{\xi}_{W^{1,2}(\Sigma)} + \norm{\eta-\eta_h}_{W^{2,2}(\Sigma)},
\end{aligned}
\end{align}
where the bounds depend linearly on $\frac{1}{\delta}, L, \frac{L}{\delta}$.
Moreover, we find 
$\PiF \hvu (\xx,1) =  (0,\Riesz \xi(\xx)) $ on $\Sigma$ and
\begin{equation}\label{P3}
\int_{\Oref} {\hq \nabla \PiF \vu : \M(\eta_h)} \, dx= 0
\quad \forall \; \hq \in \hQfh. 
\end{equation}
\end{Theorem}

\subsection{The numerical method}\label{2.2}
Now we are ready to propose a finite element method for the discretization of the weak formulation \eqref{wf_ref}. 
\begin{tcolorbox}
\begin{algorithmic}\label{alg:A}
\STATE {\bf A linear partitioned scheme.} 
\STATE For $k=0$, predict the first motion of structure by $\eta_h^1=\eta_h^0+\tau\xi_h^0.$
\STATE For $k=1,2,\cdots$, solve the FSI problem in two subproblems. 
\STATE \quad Step 1 -- The fluid subproblem: find $(\hp_h^k,\hvu_h^k)\in \hQfh\times\hVfh $ 
\begin{subequations}\label{SKM_ref}
\begin{equation}\label{SKM_ref11}
\intOref{    \hq  \Grad \hvu_h^k :  \Mhk} =0 ,
\end{equation}
\begin{equation}\label{SKM_ref21}
\begin{split}
&\vrf\intOref{ (\detJ_h^{k}\PDt \hvu_h^k 
+\frac12 \PDt \detJ_h^k \hvu_h^{k*} ) \cdot \hbfphi  } 
+\frac12 \vrf \intOref{  \big( \hbfphi \cdot (\Grad\hvu_h^k)      -  \hvu_h^k \cdot (\Grad\hbfphi)   \big)  \cdot  (\Jacob_h^k)^{-1}  \cdot \hvv_h^{k-1} \detJ_h^k } 
\\& 
+  \intOref{  \hbftau(\hvu_h^k,\hp_h^k)  : \left( \Grad \hbfphi  (\Jacob_h^k)^{-1} \right) \detJ_h^k  }
 + \frac{\vrs}{\TS}  \intS{ (\hvu_h^k - \xi_h^{k-1}\er)\cdot \hbfphi} +a_s(\eta_h^k,\zeta_h^k,\hbfphi\cdot\er) =0
\end{split}
\end{equation}
for all $(\hq,\hbfphi)\in \hQfh\times\hVfh$, where $D_t v^k :=(v^k- v^{k-1})/\TS $ is the time difference, $\hvu_h^{k*} =2 \hvu_h^{k-1} - \hvu_h^{k}$, $\hvv_h^{k-1}=\hvu_h^{k-1} - \hvw_h^k$,  and $\hbftau$ and $\M_h^k = \M(\eta_h^k)$ are given in \eqref{tauref}. Moreover, $\zeta_h$ is defined by \eqref{laph} and $a_s$ is given by \eqref{as2}. 
\STATE \quad Step 2 -- The structure subproblem: find $(\xi_h^k=\PDt \eta_h^{k+1},\eta_h^{k+1})\in V_h^s\times V_h^s$, such that
\begin{equation}\label{Structure_sub1a}
\frac{\vrs}{\TS}\intS{\xi_h^{k}\psi}+a_s(\eta_h^{k+1}, \zeta_h^{k+1},\psi)=\frac{\vrs}{\TS }\intS{\hvu_h^k\cdot \er \psi }+a_s(\eta_h^{k},\zeta_h^{k},\psi)
\end{equation}
\end{subequations}
for all  $\psi \in \Vsh$.
\end{algorithmic}
\end{tcolorbox}


\begin{Remark}
Our scheme approximates the no-slip velocity condition on the interface with a truncation error of order ${\cal O}(\TS^2)$. 
Denoting 
$\txi_h^k = \xi_h^k + \frac{\tau^2}{\vrs} \Delta_\xi^k$  with 
\begin{equation}\label{deltaxi}
    \Delta_\xi^k
=D_t (\gamma_1  \zeta_h^{k+1}  - \gamma_2 \Laph \zeta_h^{k+1} )
=  \Laph (\gamma_2  \Laph \xi_h^k - \gamma_1 \xi_h^k ) 
\end{equation}
and taking $\psi = \hvu_h^k\cdot \er - \txi_h^k \in \Vsh$ as the test function in \eqref{Structure_sub1a},  we get 
$\intS{|\hvu_h^k\cdot \er- \txi_h^k|^2 } =0.$ 
As the first component of $\hvu_h^k $ vanishes on the top boundary, we obtain the following modified no-slip condition of the velocities on the interface
\begin{equation}\label{uv}
\hvu_h^k=\txi_h^k \er.
\end{equation}

Summing up  \eqref{SKM_ref21} and \eqref{Structure_sub1a}  with the coupled test function $(\hbfphi,\psi = \widehat{\varphi}_2|_\Sigma)$ yields
\begin{multline}\label{equstability}
\vrf\intOref{ \PDt \hvu_h^k \cdot \hbfphi \detJ_h^{k} } 
+\frac12 \vrf \intOref{  \PDt \detJ_h^k \hvu_h^{k*}  \cdot \hbfphi  } 
+\frac12 \vrf \intOref{  \big( \hbfphi \cdot (\Grad\hvu_h^k)      -  \hvu_h^k \cdot (\Grad\hbfphi)   \big)  \cdot  (\Jacob_h^k)^{-1}  \cdot \hvv_h^{k-1} \detJ_h^k } 
\\
+  \intOref{  \hbftau(\hvu_h^k,\hp_h^k)  : \left( \Grad \hbfphi  (\Jacob_h^k)^{-1} \right) \detJ_h^k  }
 + \vrs  \intS{ \PDt \xi_h^k \widehat{\varphi}_2 } +a_s(\eta_h^{k+1},\zeta_h^{k+1},\widehat{\varphi}_2 )=0.
\end{multline}
\end{Remark}

\section{Stability}
\label{sec:stab}
In this section, we study the energy stability of the scheme \eqref{SKM_ref}. The total energy of the FSI system at time $t^k$ reads
$$
 E_h^k =\frac{\rho_f}{2} \intOref{ \eta_h^k   |\hvu_h^{k}|^2  }  
+ \frac{\vrs}{2}\norm{\xi_h^k}_{L^2(\Sigma)}^2 
+ \frac{\gamma_1}{2}\norm{\pdx \eta_h^{k+1}}_{L^2(\Sigma)}^2 
+ \frac{\gamma_2}{2}\norm{\Laph \eta_h^{k+1}}_{L^2(\Sigma)}^2 .
$$
\begin{Theorem}[Energy estimates]\label{Thm_Sta} 
Let  $\{(\hp_h^{k},\hvu_h^{k},\xi_h^k,\eta_h^{k+1})\}_{k=1}^{N}$ be the solution of scheme \eqref{SKM_ref}. 
We have the following energy estimates for all $m=1,\dots,N$
\begin{equation}\label{eq_Sta22}
    E_h^m +D_{num1}^m + \TS \summ \left( 2\mu \intOref{ \eta_h^k \abs{(\Grad \hvu_h^k (\Jacob_h^k)^{-1})^\rmS }^2 }+ D_{num2}^k\right)
     =     E_h^0+D_{num1}^0,
\end{equation}
where $D_{num1}^k$ and $D_{num2}^k$ are the numerical dissipation given by
\[
\begin{aligned}
D_{num1}^k=&\frac{\TS^2}{2}\left(\gamma_1 \norm{\pdx \xi_h^{k} }_{L^2(\Sigma)}^2+\gamma_2 \norm{\Laph \xi_h^{k} }_{L^2(\Sigma)}^2\right)
+ \frac{\TS^2}{2\vrs} \intS{ |\gamma_1  \zeta_h^{k+1}  - \gamma_2 \Laph \zeta_h^{k+1} |^2},
\\
D_{num2}^k=&
\frac{\TS}{2}\vrf\intOref{  \eta_h^{k-1} \abs{ D_t \hvu_h^k}^2 }
+ \frac{\vrs \TS}{2}  \norm{\PDt \xi_h^{k} }_{L^2(\Sigma)}^2+ \frac{\gamma_1 \TS }{2}  \norm{\pdx  \xi_h^k }_{L^2(\Sigma)}^2 
+ \frac{\gamma_2\TS}{2}   \norm{\Laph  \xi_h^k }_{L^2(\Sigma)}^2
\\&+ \frac{\TS^3}{2} \intS{\left(|\pdx  \PDt\xi_h^k|^2+|\Laph  \PDt \xi_h^k|^2+\frac{1}{\vrs} |\PDt(\gamma_1  \zeta_h^{k+1}  - \gamma_2 \Laph \zeta_h^{k+1} )|^2\right)}.
\end{aligned}
\]
\end{Theorem}

\begin{proof}
Setting the test functions as $(\hq,\hbfphi) = (\hp_h^k,\hvu_h^k)$ and $\psi= \xi_h^k + \frac{\TS^2}{\vrs} \Delta_\xi^k$ with $\Delta_\xi^k = \Laph (\gamma_2  \Laph \xi_h^k - \gamma_1 \xi_h^k )$ given in \eqref{deltaxi} in the numerical method and recalling $\detJ_h=\eta_h$, we get 
\begin{multline}\label{k0}
\vrf\intOref{ \PDt \hvu_h^k \cdot \hvu_h^k \eta_h^{k} } 
+\frac{\vrf}2  \intOref{  \PDt \eta_h^k \hvu_h^{k*}  \cdot \hvu_h^k } 
 + 2 \mu   \intOref{ \eta_h^k \abs{(\Grad \hvu_h^k (\Jacob_h^k)^{-1})^\rmS }^2 }  
 \\ + \vrs  \intS{ \PDt \xi_h^k \xi_h^k } +\underbrace{\TS^2\intS{ \PDt \xi_h^k \Delta_\xi^k }}_{=: \; T_1}
 +  a_s(\eta_h^{k+1}, \zeta_h^{k+1}, \xi_h^k) 
+ \underbrace{ \frac{\TS^2}{\vrs} a_s(\eta_h^{k+1}, \zeta_h^{k+1}, \Delta_\xi^k) }_{ =: \; T_2}
= 0.
\end{multline} 
Next, using the algebraic equality $(a-b)a=(a^2-b^2)/2+(a-b)^2/2$ it is easy to check the following identities:
\begin{equation*}
 \vrs\intS{ \PDt  \xi_h^k  \xi_h^k } = \frac{\vrs}2  \intS{ \PDt | \xi_h^k |^2}
+ \frac{\vrs\TS}2  \intS{ |\PDt  \xi_h^k |^2}, 
\end{equation*} 
\begin{equation*}
\begin{aligned}
\vrf\intOref{ \PDt \hvu_h^k \cdot \hvu_h^k \eta_h^{k} } 
+\frac{\vrf}2 \intOref{  \PDt \eta_h^k \hvu_h^{k*}  \cdot \hvu_h^k } 
=\vrf  \PDt\intOref{  \frac12 \eta_h^k   |\hvu_h^{k}|^2  }+\frac{\vrf\TS}{2}\intOref{  \eta_h^{k-1} \abs{ D_t \hvu_h^k}^2 },
\end{aligned}
\end{equation*}
\begin{equation*}
\begin{split}
  a_s(\eta_h^{k+1},  \zeta_h^{k+1},  \xi_h^{k})
  =& \PDt  \left( \frac{\gamma_1}{2}  \norm{ \pdx \eta_h^{k+1}}_{L^2(\Sigma)}^2 +  \frac{\gamma_2}{2}  \norm{ \Laph \eta_h^{k+1} }_{L^2(\Sigma)}^2 \right)  
\\&
+\frac{\gamma_1 \TS}{2} \intS{  \abs{\pdx \xi_h^{k}}^2 }
+ \frac{\gamma_2 \TS }{2} \intS{ \abs{ \Laph \xi_h^{k}}^2  } .
\end{split}
\end{equation*} 
Similarly, we reformulate $T_1$ and $T_2$ in the following way.
\begin{equation*}
\begin{aligned}
T_1&= \TS^2\intS{ \PDt \xi_h^k \Delta_\xi^k } =  \TS^2\intS{ \PDt \xi_h^k \Laph (\gamma_2  \Laph \xi_h^k - \gamma_1 \xi_h^k )}\\
&= \frac{\TS^2}2 \PDt  \intSB{ \gamma_2 |\Laph \xi_h^k|^2 + \gamma_1 |\pdx \xi_h^k|^2  } 
+ \frac{\TS^3}2\intSB{ \gamma_2 |\PDt \Laph \xi_h^k|^2+  \gamma_1 |\PDt \pdx \xi_h^k|^2},
\\
T_2 &=
\frac{\TS^2}{\vrs} a_s(\eta_h^{k+1}, \zeta_h^{k+1}, \Delta_\xi^k) 
= \frac{\TS^2}{\vrs} \intS{ (\gamma_1  \zeta_h^{k+1}  - \gamma_2 \Laph \zeta_h^{k+1} )  D_t (\gamma_1  \zeta_h^{k+1}  - \gamma_2 \Laph \zeta_h^{k+1} )}
\\  &= 
 \frac{\TS^2}{2\vrs} \PDt \intS{ |\gamma_1  \zeta_h^{k+1}  - \gamma_2 \Laph \zeta_h^{k+1} |^2}
 + \frac{\TS^3}{2\vrs} \intS{ |\PDt(\gamma_1  \zeta_h^{k+1}  - \gamma_2 \Laph \zeta_h^{k+1} )|^2}.
\end{aligned}
\end{equation*}
Consequently, collecting the above equalities and substituting them into \eqref{k0}, we derive
\begin{equation}\label{k3}
    \PDt (E_h^k + D_{num1}^k) +2\mu \intOref{ \eta_h^k \abs{(\Grad \hvu_h^k (\Jacob_h^k)^{-1})^\rmS }^2 } + D_{num2}^k = 0.
\end{equation}
Finally, computing $\TS\, \summ \eqref{k3}$ yields \eqref{eq_Sta22}, 
which completes the proof. 
\end{proof}
To ensure that the energy remains meaningful (i.e., positive), we assume that $\eta_h^k > 0$.  More precisely, we assume that for all $t^k \in [0, T]$
 \begin{equation}\label{noc}
\eta_h \geq \underline{\eta}>0. 
\end{equation}
Note that this assumption can be justified by mathematical induction via the error estimates, see our previous result \cite[Theorem~6.2]{SST}. 
From Theorem~\ref{Thm_Sta} and \eqref{noc}, we have the following uniform estimates.  
\begin{Corollary}\label{co_est}
Let the initial data satisfy $\hvu_0 \in W^{1,2}( \Oref;\R^2), \; \eta_0 \in W^{1,2}(\Sigma), \text{ and }  \xi_0 \in W^{1,2}(\Sigma)$. 
Let $(\hp_h,\hvu_h,\xi_h,\eta_h)  = \{(\hp_h^{k},\hvu_h^{k},\xi_h^{k},\eta_h^{k+1})\}_{k=1}^{N}$ be a  solution to scheme \eqref{SKM_ref} with $(\TS,h ) \in (0,1)^2$ and let \eqref{noc} hold.  Then we have the following uniform bounds.
\begin{equation}\label{ests}
\begin{aligned}
 &
\norm{\xi_h}_{L^\infty (0,T;L^2(\Sigma))} 
+ \norm{\pdx \eta_h}_{L^\infty (0,T;L^2(\Sigma))} 
+ \norm{\Laph \eta_h}_{L^\infty (0,T;L^2(\Sigma))}  \aleq 1,
\\& 
\norm{\eta_h^{-1}}_{L^\infty((0,T)\times \Sigma)} +
 \norm{\eta_h}_{L^\infty((0,T)\times \Sigma)} + 
\norm{\pdx \eta_h}_{L^\infty((0,T)\times \Sigma)} \aleq 1,
\\& 
\norm{\Jacob_h}_{L^\infty((0,T)\times \Sigma;\R^{2 \times 2})} 
+ \norm{\Jacob_h^{-1}}_{L^\infty((0,T)\times \Sigma;\R^{2 \times 2})}  
\aleq 1 ,
\\& 
\norm{\hvu_h}_{L^\infty (0,T;L^2(\Oref))} + 
\norm{(\Grad \hvu_h (\Jacob_h)^{-1})^\rmS  }_{L^2((0,T)\times \Oref)} \aleq 1 ,
\\& 
  \norm{\Grad \hvu_h }_{L^2((0,T)\times \Oref)} \aleq 1, \quad 
  \norm{\hvu_h}_{L^2(0,T;L^{q_1}(\Oref))} \aleq 1,
\\&
   \norm{ \xi_h}_{L^2 (0,T; L^\infty (\Sigma))} \aleq \norm{\Grad \hvu_h }_{L^2((0,T)\times \Oref)} \aleq 1 ,
\\&
  \norm{\hvw_h}_{L^\infty (0,T;L^2(\Oref))} +  \norm{\hvw_h}_{L^2(0,T;L^\infty (\Oref))} \aleq  1,
\\&
    \norm{\hvv_h}_{L^2(0,T;L^{q_1}(\Oref))} 
 +  \norm{\hvv_h}_{L^\infty (0,T;L^2(\Oref))} \aleq 1, \quad 
 \norm{\hvv_h}_{L^{q_2}(0,T;L^{q_1}(\Oref;\R^2))}\aleq 1 
.
\end{aligned}
\end{equation} 
for any $q_1\in [1,\infty)$ and  $q_2 \in [1,\infty)$. 
\end{Corollary}

\section{Error estimates}
In this section, we study the error between the numerical solution $(\hp_h,\hvu_h,\xi_h,\eta_h)$ of scheme \eqref{SKM_ref} and its target smooth solution $(\hp,\hvu,\xi,\eta)$. Here we assume the existence of a smooth solution of \eqref{pde_f}--\eqref{pde_bdc} in the following class 
\begin{equation}\label{STClass}
\left\{
\begin{aligned}
& \eta>\underline{\eta}, \; \eta \in 
W^{1,2}(0,T; W^{4,2}(\Sigma))\cap  W^{2,2}(0,T; W^{2,2}(\Sigma)),
\\&  
\hvu \in   
L^\infty (0,T; W^{1,2}(\Oref;\R^2)) \cap 
L^2 (0,T; W^{2,2}(\Oref;\R^2))
\\&
\pdt \hvu \in  L^2 (0,T;W^{1,2}( \Oref;\R^2 )) \cap  L^\infty (0,T;L^2( \Oref;\R^2 )), 
\pdt ^2 \hvu \in  L^2 ((0,T)\times\Oref;\R^2 ) 
\; 
\\&
\hp \in L^\infty( 0,T; L^2(\Oref) ), \; \Grad p \in L^2 ( (0,T)\times \Oref ). 
\end{aligned}
\right.
\end{equation}

Before introducing the main result, let us denote the following error terms for each time step $k \in\{1,\dots,N_T\}$.
\begin{equation}\label{ers}
\begin{aligned}
&e_p^k 	=   \hp_h^k  -\hp^k   =(\hp_h^k  	- \Piq \hp^k)  	+ (\Piq \hp^k -  \hp^k )      	=: \delta_p^k 	+ I_p^k , \\ 
&\eu^k  	= \hvu_h^k -\hvu^k  =(\hvu_h^k 	- \PiF \hvu^k) 	+ (\PiF \hvu^k -  \hvu^k )   		=: \delta_\vu^k 	+ I_\vu^k ,   \\ 
&e_\xi^k   	=\xi_h^k -\xi^k   	= (\xi_h^k 	- \Riesz  \xi^k )  	+ (\Riesz  \xi^k -  \xi^k)    		=: \delta_\xi^k 	+ I_\xi^k ,\\
&e_\eta^k  	= \eta_h^k -\eta^k	=(\eta_h^k	- \Riesz \eta^k)	+ (\Riesz \eta^k - \eta^k ) 	=: \delta_\eta^k	+ I_\eta^k ,  \\ 
&e_\zeta^k  	= \zeta_h^k -\zeta^k	=(\zeta_h^k	+ \Laph \Riesz \eta^k)	+ (-\Laph \Riesz \eta^k - \zeta^k )	=: \delta_\zeta^k	+ I_\zeta^k .  \\ 
\end{aligned} 
\end{equation}
Now we are ready to present the main result of the paper. 
\begin{Theorem}[Convergence rate]\label{theorem_conv_rate}
Let $\{(\hp_h^k, \hvu_h^k,\xi_h^k,\eta_h^{k+1})\}_{k=1}^{N_T}$ be the solution of the splitting scheme \eqref{SKM_ref}, and let  $(\hvu,\hp,\xi,\eta)(t)$, $t\in(0,T)$, be a strong  solution of \eqref{pde_f}--\eqref{pde_bdc} belonging to the class \eqref{STClass}. 
Then for any $m \in \{1,\cdots,N_T\}$ it holds
\begin{multline*}
 \frac12 \vrf \intOref{|e_\vu^m|^2 \eta_h^m}   + \frac12  \intSB{ \vrs \abs{e_\xi^m}^2 +  \gamma_1 |\pdx e_\eta^{m+1} |^2 +  \gamma_2 \abs{ e_\zeta^{m+1}}^2}
\\ +  2 \mu   \TS \summ \intOref{\left| \Grad e_\vu^k (\Jacob_h^k)^{-1}\right|^2} 
\aleq \TS^2 +h^2.
\end{multline*}
In particular, we have the following convergence rates
\begin{multline*}
 \norm{e_\vu}_{L^\infty(0,T;L^2(\Oref;\R^2))}
+\norm{e_\xi}_{L^\infty(0,T;L^2(\Sigma))}
+\norm{\pdx e_\eta}_{L^\infty(0,T;L^2(\Sigma))}
+\norm{e_\zeta}_{L^\infty(0,T;L^2(\Sigma))}
\\
+\norm{\Grad e_\vu}_{L^2((0,T)\times \Oref;\R^{2\times 2})}
\aleq 
\TS +h .
\end{multline*}
\end{Theorem}

\begin{proof}
First, we subtract the weak formulation \eqref{wf2_ref} from the numerical scheme \eqref{equstability} and get 
\begin{equation}\label{EEQ}
\begin{aligned}
& \intOref{  \vrf (\eta_h^k \PDt \euk  + \frac12 \PDt \eta_h^k \eu ^{k*}  ) \cdot \hbfphi  }  
+ 2 \mu \intOref{ \big( \Grad \euk  (\Jacob_h^k)^{-1} \big)^\rmS: \big(\Grad \hbfphi (\Jacob_h^k)^{-1} \big)\eta_h^k  } 
\\& 
+ \vrs\intS{\PDt e_\xi^k \psi}  
+a_s(e_\eta^{k+1},e_\zeta^{k+1},\psi)
= - \sum_{i=1}^7 R^k_i( \hbfphi, \psi),
\end{aligned}
\end{equation}
see Appendix~\ref{app_ee1} for the details.  Here,  
\begin{equation}\label{RS}
\begin{aligned}
&R^k_1( \hbfphi, \psi)=  \vrf \intOref{\big( e_\eta^k \pdt \hvu^k+ \eta_h^k (\PDt \hvu^k -\pdt\hvu^k ) \big) \cdot \hbfphi},
\\ & 
R^k_2( \hbfphi, \psi) = \frac12\vrf \intOref{\left( (e_\xi^{k-1} - \TS \PDt \xi^k )\hvu^{k*}- \TS \pdt \eta^k \PDt \hvu^k \right) \cdot \hbfphi},
\\& 
R^k_3( \hbfphi, \psi)=\frac12 \vrf \intOref{ \Big( \hbfphi \cdot   (\Grad \euk ) 
-\euk   \cdot   (\Grad \hbfphi) \Big)
 \cdot  (\Jacob_h^k)^{-1} \hvv_h^{k-1}  \eta_h^k}
\\& \quad 
+ \frac12 \vrf \intOref{ \Big( \hbfphi \cdot (\Grad \hvu^k) - \hvu^k \cdot (\Grad\hbfphi) \Big) \cdot \left(    (\Jacob_h^k)^{-1} \hvv_h^{k-1}  \eta_h^k-   (\Jacob^k)^{-1}\hvv^{k} \eta^k \right)  }, 
 \\& 
R^k_4( \hbfphi, \psi) 
=\intOref{  e_p^k \Grad \hbfphi : \Mhk } 
+ \intOref{ \hp^k \Grad \hbfphi : \big(\Mhk -  \Mk\big)} ,
\\& 
R^k_5( \hbfphi, \psi)=2 \mu \intOrefB{\big( \Grad \hvu (\Jacob_h^k)^{-1} \big)^\rmS: (\Grad \hbfphi (\Jacob_h^k)^{-1} )\eta_h^k
-\big( \Grad \hvu (\Jacob)^{-1} \big)^\rmS: (\Grad \hbfphi (\Jacob^k)^{-1} )\eta^k} , 
\\& 
R^k_6( \hbfphi, \psi)=\vrs\intS{(\PDt \xi^k -\pdt \xi^k  ) \psi} ,
\\&
R^k_7( \hbfphi, \psi)= - \gamma_1 \intS{\Lapx (\eta^{k+1} -\eta^k) \; \psi}
- \gamma_2 \intS{\Lapx (\zeta^{k+1} -\zeta^k)\; \psi}.
\end{aligned}
\end{equation}

Note that on $\Sigma$ it holds
\begin{equation}\label{verr}
\begin{aligned}
\delta_\vu^k &=\hvu_h^k-\PiF \hvu^k = \txi_h^k \er -\PiF\hvu^k= (\xi_h^k  + \frac{\TS^2}{\vrs}\Delta_\xi^k)\er -\PiF\hvu^k= (\delta_\xi^k  +\frac{\TS^2}{\vrs}\Delta_\xi^k)\er, \\
\end{aligned}
\end{equation}
where we have used the property of the projection $\PiF$ that
\begin{equation*}
\Riesz\xi^k\er=\PiF\hvu^k|_\Sigma.
\end{equation*}
We can now proceed with the proof by setting $\hbfphi = \delta_\vu^k$ and $
\psi=\delta_\xi^k  + \frac{\TS^2}{\vrs}\Delta_\xi^k
$  in \eqref{EEQ}. Note that they are a pair of admissible test functions since \eqref{verr} yields $\hbfphi|_\Sigma=\psi\er$.
Then sum up from $k=1$ to $m$ we derive 
\begin{equation*}
\begin{aligned}
-&  \TS \summ  \sum_{i=1}^7 R^k_i( \delta_\vu^k, \delta_\xi^k  + \frac{\TS^2}{\vrs}\Delta_\xi^k)
\\ =&  \TS \summ  \intOref{  \vrf (\eta_h^k \PDt \euk  + \frac12 \PDt \eta_h^k \eu ^{k*}  ) \cdot \delta_\vu^k  }  
+ 2 \mu  \TS \summ  \intOref{ \big( \Grad \euk  (\Jacob_h^k)^{-1} \big)^\rmS: \big(\Grad \delta_\vu^k (\Jacob_h^k)^{-1} \big)\eta_h^k  } 
\\& 
+  \TS \summ  \vrs\intS{\PDt e_\xi^k (\delta_\xi^k  + \frac{\TS^2}{\vrs}\Delta_\xi^k)}  
+ \TS \summ  a_s(e_\eta^{k+1},e_\zeta^{k+1},\delta_\xi^k  + \frac{\TS^2}{\vrs}\Delta_\xi^k).
\end{aligned}
\end{equation*}

Further, applying  \eqref{IM6} to the above right-hand-side, we reformulate the above equality as 
\begin{equation}\label{REI1}
\begin{aligned}
& -  \TS \summ  \sum_{i=1}^7 R^k_i( \delta_\vu^k, \delta_\xi^k  + \frac{\TS^2}{\vrs}\Delta_\xi^k)
=
 \TS \summ \intOref{  \vrf \big(\eta_h^k \PDt (\delta_\vu^k+I_\vu^k)  + \frac12 \PDt \eta_h^k (\delta_\vu^{k*}+I_\vu^{k*})  \big) \cdot \delta_\vu^k  }  
\\& \quad 
+ 2 \mu  \TS \summ \intOref{   \big( \Grad (\delta_\vu^k+I_\vu^k)  (\Jacob_h^k)^{-1} \big)^\rmS: (\Grad \delta_\vu^k (\Jacob_h^k)^{-1} ) \eta_h^k }  
\\
&+  \TS \summ  \vrs\intS{\PDt (\delta_\xi^k+I_\xi^k) (\delta_\xi^k  +\frac{\TS^2}{\vrs}\Delta_\xi^k)} + \TS \summ  a_s(e_\eta^{k+1},e_\zeta^{k+1},\delta_\xi^k  + \frac{\TS^2}{\vrs}\Delta_\xi^k)
\\& = \delta_E^m - \delta_E^0  + \TS \summ  \delta_D^k+ \TS \summ \delta_N^k + G_f+ G_s+G_{sL},
\end{aligned}
\end{equation}
where 
\begin{align}
\delta_E^k =&   \intOref{ \frac12 \vrf \eta_h^k  |\delta_{\vu}^k|^2  } 
  + \frac12 \intSB{  \vrs |\delta_\xi^k|^2+ \gamma_1 |\pdx \delta_\eta^{k+1}|^2  
	+ \gamma_2|\delta_\zeta^{k+1}|^2 },
\notag\\
 \delta_D^k = &   2 \mu \intOref{ \eta_h^k |\big(\Grad \delta_{\vu}^k (\Jacob_h^k)^{-1}\big)^\rmS|^2 },
\notag\\
 \delta_N^k = &
 \frac{\TS}{2}\vrf  \intOref{  \eta_h^{k-1}  |\PDt \delta_{\vu}^k|^2  } 
 +	 \frac{\TS}{2} \intSB{\vrs|\PDt \delta_\xi^k|^2+ \gamma_1|\PDt \pdx \delta_\eta^{k+1}|^2+\gamma_2 |\PDt \delta_\zeta^{k+1}|^2} \geq 0,
\notag\\
G_f =&
\TS \summ  \intOref{  \vrf \big(\eta_h^k \PDt I_\vu^k  + \frac12 \PDt \eta_h^k I_\vu^{k*} \big) \cdot \delta_\vu^k  }  
\notag\\&  + 2 \mu \TS \summ   \intOref{   \big( \Grad I_\vu^k  (\Jacob_h^k)^{-1} \big)^\rmS: (\Grad \delta_\vu^k (\Jacob_h^k)^{-1} ) \eta_h^k }  ,
\notag\\ G_s=&
\gamma_1 \TS \summ \intS{\pdx \delta_\eta^{k+1} \pdx (\PDt \eta^{k+1} - \pdt\eta^k )} -
\gamma_2 \TS \summ\intS{ \delta_\zeta^{k+1} \Lapx (\PDt \eta^{k+1} - \pdt\eta^k )}
\notag\\&+\TS \summ  \intS{\vrs \PDt I_\xi^k  (\delta_\xi^k  + \frac{\TS^2}{\vrs}\Delta_\xi^k)},
\notag\\  
G_{sL}=&  \TS^3 \summ \left( \intS{\PDt \delta_\xi^k\Delta_\xi^k}+  \frac{1}{\vrs}a_s(\delta_\eta^{k+1},\delta_\zeta^{k+1}, \Delta_\xi^k)\right).\label{REFB5}
\end{align}
Next, we reformulate \eqref{REI1} in the following form. 
\begin{equation}\label{REI}
\delta_E^m - \delta_E^0
+  \TS \summ  \delta_D^k
   +\TS\summ  D_{num}^k+G_{sL}
=  - \TS \summ \sum_{i=1}^7 R^k_i -G_f - G_s.
\end{equation}

Then, by Young's inequality, H\"older's inequality, the interpolation error in Theorem~\ref{thm:projection-velocity}, and the uniform bounds \eqref{ests}, we estimate the right-hand-side of the above equation as
\begin{multline}\label{res}
\Abs{ \TS \summ \sum_{i=1}^7 R^k_i + G_f + G_s} 
\aleq \TS^2 + h^2  
+ c \TS \summ \delta_E^k 
+ 2 \alpha  \mu   \TS \summ \intOref{\left| \Grad \delta_\vu^k (\Jacob_h^k)^{-1}\right|^2 \eta_h^k}
\\ +\TS^3 \summ  \Big( \norm{\delta_\zeta^{k+1} }_{L^2(\Sigma)}^2+\norm{\Laph \delta_\zeta^{k+1}}_{L^2(\Sigma)}^2 \Big),
\end{multline}
see Appendix \ref{app_res}. 

Further, substituting the above estimate into \eqref{REI} and noticing the initial error $\delta_E^0=0$, owing to the estimate of $G_{sL}$ stated in Lemma \ref{b5}, the lower bounds of $\eta$ and $\eta_h$, and by denoting
$$
\widetilde{\delta}_E^k=\delta_E^k+\frac{\tau^2\gamma_1^2}{2\vrs}\norm{\delta_\zeta^{m+1}}^2_{L^2(\Sigma)} +\frac{\tau^2\gamma_2^2}{2\vrs}\norm{\Laph\delta_\zeta^{m+1}}^2_{L^2(\Sigma)},
$$
we get 
\begin{align*}
\widetilde{\delta}_E^m  + (1&-\alpha) 2 \mu \TS \summ \intOref{\left| \Grad \delta_\vu^k (\Jacob_h^k)^{-1}\right|^2 \eta_h^k }\aleq \TS^2 +h^2 + \TS \summ \widetilde{\delta}_E^k. 
\end{align*}
By choosing any $\alpha\in(0,1)$ and using Gronwall's inequality, we get
\begin{align*}
\delta_E^m  +  \TS \summ \delta_D^k\leq \widetilde{\delta}_E^m +  \TS \summ \delta_D^k
\aleq \TS^2 +h^2 . 
\end{align*}
Recalling the interpolation errors (see Section \ref{Sec_wf}) and the regularity of the strong solution \eqref{STClass} we get
\begin{multline*}
 \frac12 \vrf \intOref{|I_\vu^m|^2 \eta_h^m}   + \frac12  \intSB{ \vrs \abs{I_\xi^m}^2 +  \gamma_1 |\pdx I_\eta^{m+1} |^2 +  \gamma_2 \abs{ I_\zeta^{m+1} }^2}
\\ +     \TS \summ \left( 2 \mu\intOref{\left| \Grad I_\vu^k (\Jacob_h^k)^{-1}\right|^2 \eta_h^k }  \right)
 \aleq  h^2.
\end{multline*}

Finally, due to the triangular inequality, we sum up the previous two estimates and get
\begin{multline*}
 \frac12 \vrf \intOref{|e_\vu^m|^2 \eta_h^m}   + \frac12  \intSB{ \vrs \abs{e_\xi}^2 +  \gamma_1 |\pdx e_\eta^{m+1} |^2 +  \gamma_2 \abs{ e_\zeta^{m+1}}^2}
\\ +     \TS \summ \left( 2 \mu\intOref{\left| \Grad e_\vu^k (\Jacob_h^k)^{-1}\right|^2 \eta_h^k }  \right)
\quad \aleq
\TS^2 +h^2.
\end{multline*}
Note that the above proof is valid on the time interval $(0, T)$ under the assumption $\eta_h^k > \underline{\eta} > 0$. Following the methodology in the proof of \cite[Theorem~6.2]{SST}, one may deduce via the mathematical induction that if $\eta_h(0) > \underline{\eta}$, then \eqref{noc} remains valid as long as the smooth solution exists and satisfies the lower bound $\eta > \underline{\eta}$.
\end{proof}
\section{Numerical experiments}
\label{sec:num}
In this section, we evaluate the numerical performance of scheme \eqref{SKM_ref} and verify the theoretical convergence rate. The scheme is implemented using the Firedrake finite element package \cite{Firedrake}; 
the code is publicly available in \cite{zenodo:code}.

\subsection{Experiment 1: Large deformation driven by external force}

The computational domain $\Oref$ is a rectangle of size $2\times1$ with periodic boundary conditions in the $\xx$-direction. 
At the bottom boundary, we impose the no-slip condition $\mathbf{u} = 0$. 
At the top boundary, the velocity in the $x_1$-direction is set to zero, while in the $x_2$ direction a time-dependent force is applied. 
All unknowns are initialized to zero at $t=0$. 
We set $\mu = 0.01$, $\vrf = \vrs = 1$, and $\gamma_1 = \gamma_2 = 0.1$. 
The flow is driven by an external force $g$ periodic in $x_1$, applied at the top boundary up to $t = 0.2$ to generate a large structural deformation. 
After $t = 0.2$, the force is switched off and the system relaxes. 
The force is given by
$$
g(x_1, t) = 
\begin{cases}
200\, t \sin(2\pi x_1), & t \le 0.2,\\[4pt]
0, & t > 0.2.
\end{cases}
$$
Snapshots of the simulation are shown in Figure~\ref{fig:snapshots}.

\begin{figure}[!htbp]
\centering
\includegraphics[width=5.4cm]{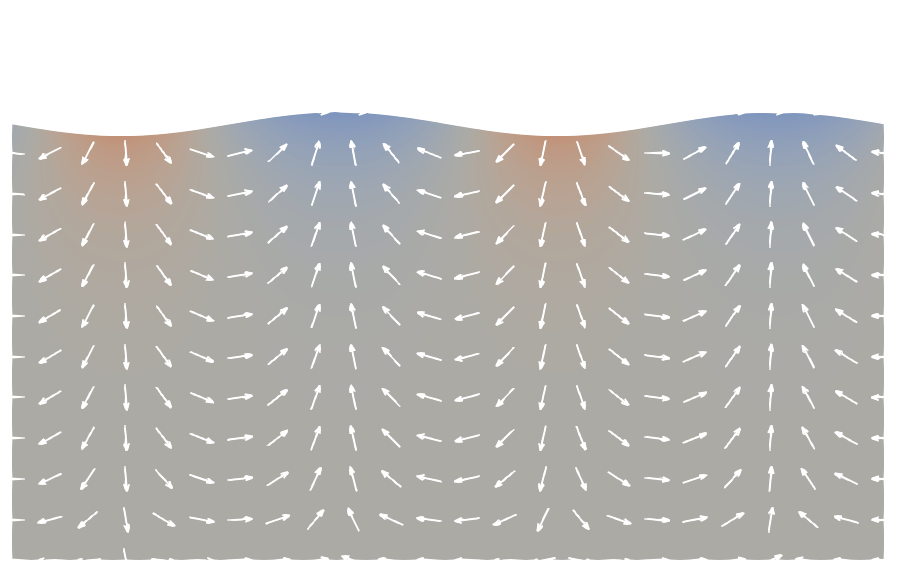}\hspace*{0.4cm}
\includegraphics[width=5.4cm]{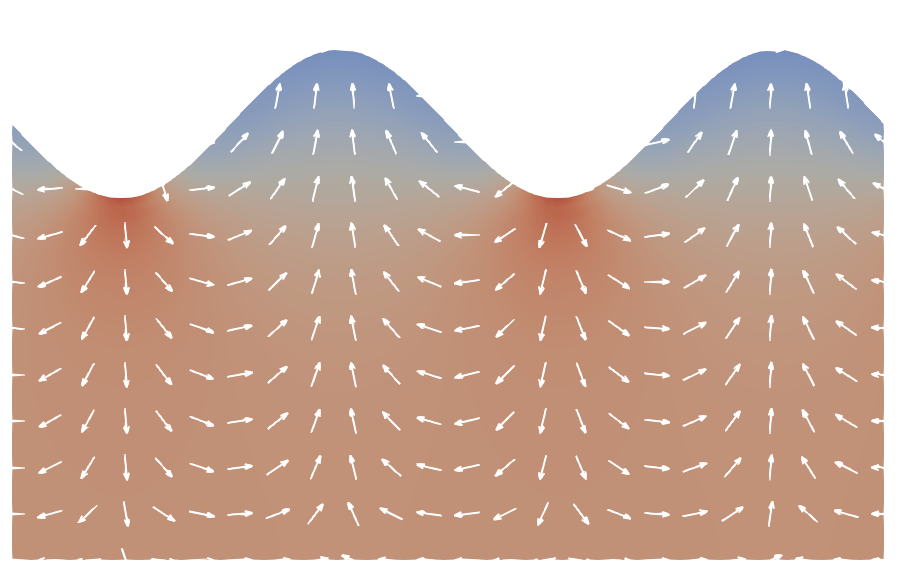}\hspace*{0.4cm}
\includegraphics[width=5.4cm]{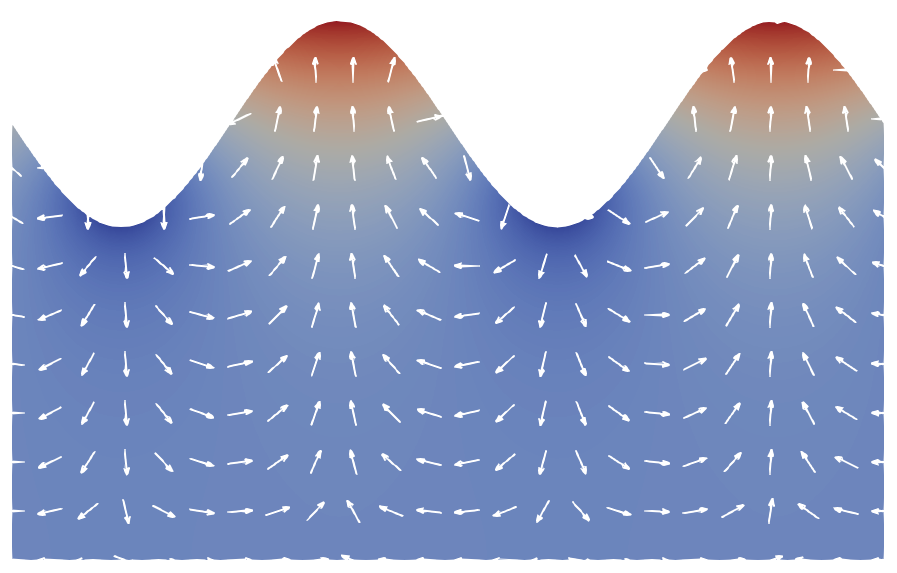}\\[-1pt]
$t=0.1$\hspace*{4.8cm}$t=0.2$\hspace*{4.8cm}$t=0.25$\\[5pt]
\includegraphics[width=5.4cm]{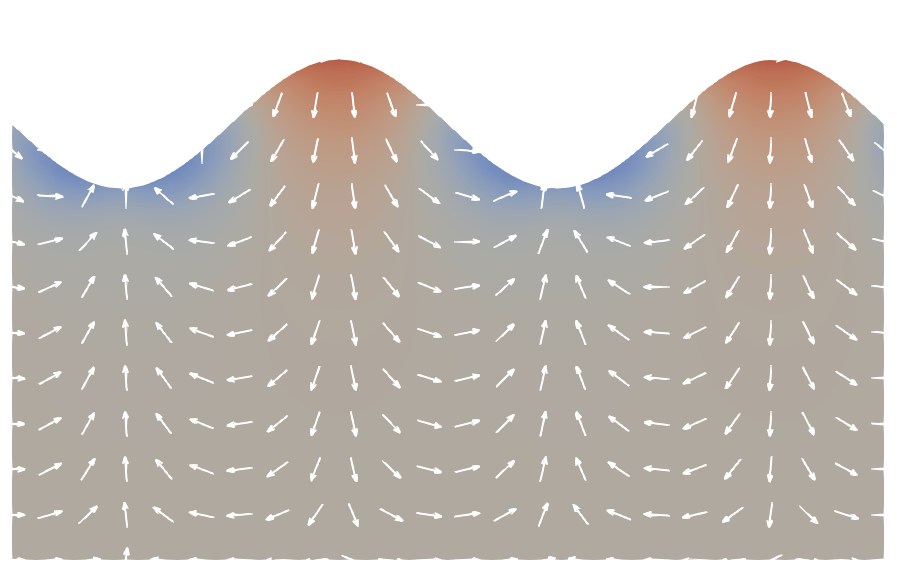}\hspace*{0.4cm}
\includegraphics[width=5.4cm]{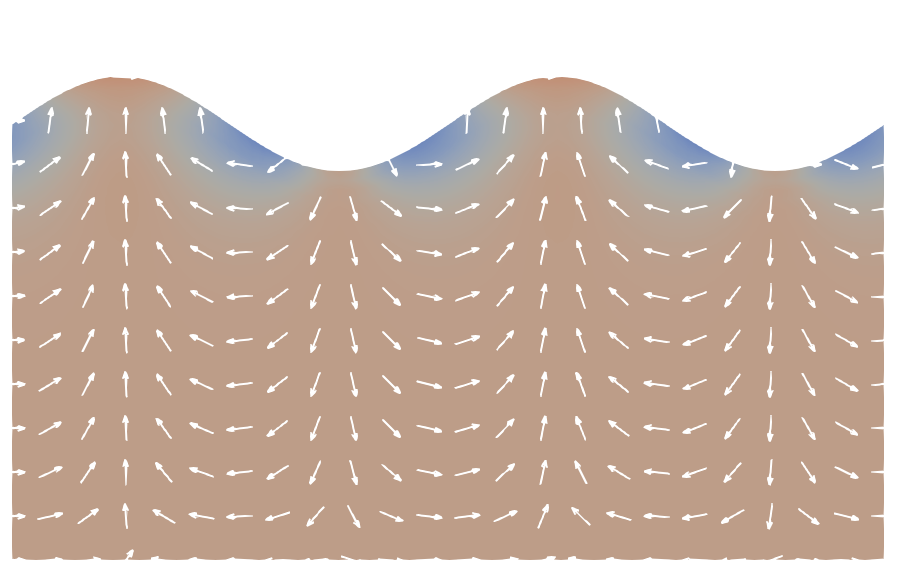}\hspace*{0.4cm}
\includegraphics[width=5.4cm]{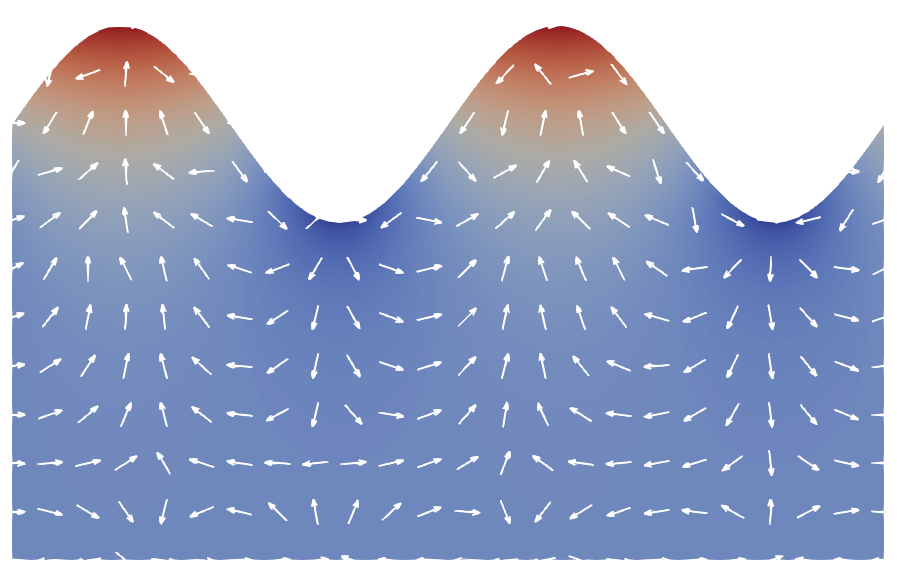}\\[-1pt]
$t=0.35$\hspace*{4.8cm}$t=0.45$\hspace*{4.8cm}$t=0.55$\\[5pt]
\caption{Snapshots of the simulation at different time instants. The color scale depicts pressure, arrows show the direction of the velocity field.}
\label{fig:snapshots}
\end{figure}

\subsection{Experiment 2: Convergence rates}

Simulations are carried out for $t \in [0,1]$ using six time steps 
$\TS = 5\times10^{-3},\, 2.5\times10^{-3},\, 1.25\times10^{-3},\, 6.25\times10^{-4},\, 3.125\times10^{-4}$, 
and $\TS_{\rm min} = 10^{-4}$, 
on six meshes with characteristic sizes 
$h = 2.83\times10^{-1},\, 1.41\times10^{-1},\, 7.07\times10^{-2},\, 3.54\times10^{-2},\, 1.77\times10^{-2}$, 
and $h_{\rm min} = 8.84\times10^{-3}$. 
The solution computed on the finest mesh with the smallest time step is taken as the reference solution.

For each mesh refinement and time step, the numerical solutions are compared to the reference solution by evaluating all terms appearing on the right-hand side of Theorem~\ref{theorem_conv_rate}, namely
$$
\|e_\vu\|_{L^\infty(L^2)}, \quad 
\|e_{\xi}\|_{L^\infty(L^2)}, \quad 
\|e_{\eta}\|_{L^\infty(L^2)}, \quad 
\|\nabla e_{\eta}\|_{L^\infty(L^2)}, \quad 
\|e_{\zeta}\|_{L^\infty(L^2)}, \quad 
\|\nabla e_\vu\|_{L^2(L^2)}.
$$

The convergence with respect to the mesh size $h$ (using the finest time step) is summarized in Table~\ref{tab:optimal-bilaplace_h}, 
while the convergence with respect to the time step $\TS$ ( using the finest mesh) is presented in Table~\ref{tab:optimal-bilaplace_t}.
Corresponding convergence plots are shown in Figure~\ref{fig:conv_h_t}. 
The results indicate linear convergence with respect to the time step for all error quantities, 
and quadratic convergence in space for 
$\|e_\vu\|_{L^\infty(L^2)}$, 
$\|e_{\xi}\|_{L^\infty(L^2)}$, 
$\|e_{\eta}\|_{L^\infty(L^2)}$, and 
$\|e_{\zeta}\|_{L^\infty(L^2)}$, 
while 
$\|\nabla e_\vu\|_{L^2(L^2)}$ and 
$\|\nabla e_{\eta}\|_{L^\infty(L^2)}$ 
exhibit linear spatial convergence. 

Hence, for this particular example, the observed convergence rates are higher than those predicted by Theorem~\ref{theorem_conv_rate}, which establishes only linear convergence in both space and time.

\pgfkeys{
/pgf/number format/.cd,
sci,
sci zerofill,
sci generic={mantissa sep=\times,exponent={10^{#1}}}}

\pgfplotstableset{
create on use/uLiL2quotlong/.style={create col/expr={\thisrow{uLiL2}/1.17e+00}},
create on use/xiLiL2quotlong/.style={create col/expr={\thisrow{xiLiL2}/2.75e-00}},
create on use/etaLiL2quotlong/.style={create col/expr={\thisrow{etaLiL2}/2.17e-01}},
create on use/gradetaLiL2quotlong/.style={create col/expr={\thisrow{gradetaLiL2}/1.37e-00}},
create on use/LapetaLiL2quotlong/.style={create col/expr={\thisrow{LapetaLiL2}/8.99e-00}},
create on use/graduL2L2quotlong/.style={create col/expr={\thisrow{graduL2L2}/1.22e+01}},
create on use/timeuLiL2quotlong/.style={create col/expr={\thisrow{uLiL2}/2.51e-01}},
create on use/timexiLiL2quotlong/.style={create col/expr={\thisrow{xiLiL2}/5.49e-01}},
create on use/timeetaLiL2quotlong/.style={create col/expr={\thisrow{etaLiL2}/4.20e-02}},
create on use/timegradetaLiL2quotlong/.style={create col/expr={\thisrow{gradetaLiL2}/2.64e-01}},
create on use/timeLapetaLiL2quotlong/.style={create col/expr={\thisrow{LapetaLiL2}/1.66e-00}},
create on use/timegraduL2L2quotlong/.style={create col/expr={\thisrow{graduL2L2}/1.59e-00}},
create on use/timegraduL2L2quotcomp/.style={create col/expr={\thisrow{graduL2L2}/1.60e+00}},
create on use/timegradetaLiL2quotcomp/.style={create col/expr={\thisrow{gradetaLiL2}/2.65e-01}},
columns/h/.style={int detect,column type=c, fixed zerofill, precision=2, column type/.add={|}{|}, column name=$h$},
columns/dt/.style={int detect,column type=c, fixed zerofill, precision=2, column type/.add={|}{|}, column name=$\TS$},
columns/uLiL2/.style={int detect,column type=c, fixed zerofill, precision=2, column type/.add={}{|}, column name=$\|e_\vu\|_{L^\infty(L^2)}$},
columns/xiLiL2/.style={int detect,column type=c, fixed zerofill, precision=2, column type/.add={}{|}, column name=$\|e_{\xi}\|_{L^\infty(L^2)}$},
columns/etaLiL2/.style={int detect,column type=c, fixed zerofill, precision=2, column type/.add={}{|}, column name=$\|e_{\eta}\|_{L^\infty(L^2)}$},
columns/gradetaLiL2/.style={column type=c, fixed zerofill, precision=2, column type/.add={}{|}, column name=$\|\nabla e_{\eta}\|_{L^\infty(L^2)}$},
columns/LapetaLiL2/.style={int detect,column type=c, fixed zerofill, precision=2, column type/.add={}{|}, column name=$\|e_{\zeta}\|_{L^\infty(L^2)}$},
columns/graduL2L2/.style={int detect,column type=c, fixed zerofill, precision=2, column type/.add={}{|}, column name=$\|\nabla e_\vu\|_{L^2(L^2)}$},
columns/gradxiL2L2/.style={int detect,column type=c, fixed zerofill, precision=2, column type/.add={}{|}, column name=$\|\nabla e_{\xi}\|_{L^2(L^2)}$},
empty cells with={--}, 
every head row/.style={before row=\hline,after row=\hline},
every last row/.style={after row=\hline},
}

\newcommand{\Convergence}[3]{
\pgfplotsextra{
\pgfmathsetmacro{\ax}{0.5}
\pgfmathsetmacro{\ay}{1}
\pgfmathsetmacro{\bx}{1.5}
\pgfmathsetmacro{\by}{(3.0^#3)}
\pgfmathsetmacro{\slope}{(3.0^#3)}
\coordinate (a) at (axis direction cs:\ax*#1,\ay*#2);
\coordinate (b) at (axis direction cs:\bx*#1,\by*#2);
\draw (a) -- (b) (a) -| (b) node [pos=0.25,anchor=north] {\small 1} node [pos=0.95,anchor=west] {\small order #3};
}
} 
\begin{table}[!htbp]
\caption{Convergence of errors with mesh refinement (using fixed time step $\TS=\TS_{\rm min}$); reference solution: $h_{\rm min}=8.84\times 10^{-3}$, $\TS_{\rm min}=10^{-4}$.}
\label{tab:optimal-bilaplace_h}
  \begin{center}{
    \begin{filecontents*}{optimal-h_errors_bilaplace-Gamma.txt}
h uLiL2 xiLiL2 etaLiL2 gradetaLiL2 LapetaLiL2 graduL2L2
2.83e-01 1.17e+00 2.75e+00 2.17e-01 1.37e+00 8.99e+00 1.22e+01
1.41e-01 3.16e-01 5.63e-01 5.83e-02 3.69e-01 2.35e+00 7.48e+00
7.07e-02 1.07e-01 1.35e-01 1.47e-02 1.34e-01 5.82e-01 4.11e+00
3.54e-02 2.89e-02 3.20e-02 3.51e-03 6.57e-02 1.38e-01 2.12e+00
1.77e-02 6.44e-03 6.43e-03 7.11e-04 2.94e-02 2.78e-02 1.03e+00
\end{filecontents*}
\pgfplotstabletypeset[font={\small}]{optimal-h_errors_bilaplace-Gamma.txt}}
\end{center}
\end{table}
\begin{table}[!htbp]
\caption{Convergence of errors with time step refinement (using fixed mesh size $h=h_{\rm min}$); reference solution: $h_{\rm min}=8.84\times 10^{-3}$, $\TS_{\rm min}=10^{-4}$.}
\label{tab:optimal-bilaplace_t}
  \begin{center}{
    \begin{filecontents*}{optimal-t_errors_bilaplace-Gamma.txt}
dt uLiL2 xiLiL2 etaLiL2 gradetaLiL2 LapetaLiL2 graduL2L2
5.00e-03 2.51e-01 5.49e-01 4.20e-02 2.64e-01 1.66e+00 1.59e+00
2.50e-03 1.34e-01 2.85e-01 2.19e-02 1.38e-01 8.66e-01 8.43e-01
1.25e-03 6.75e-02 1.40e-01 1.08e-02 6.80e-02 4.29e-01 4.22e-01
6.25e-04 3.15e-02 6.54e-02 5.02e-03 3.16e-02 1.99e-01 1.97e-01
3.12e-04 1.29e-02 2.68e-02 2.05e-03 1.29e-02 8.11e-02 8.06e-02
\end{filecontents*}
\pgfplotstabletypeset[font={\small}]{optimal-t_errors_bilaplace-Gamma.txt}}
\end{center}
\end{table}

\begin{figure}[!htbp]
\centering
\begin{minipage}{0.49\textwidth}
\begin{tikzpicture}[scale=0.99]
\begin{loglogaxis}[
    width=\textwidth,
    ylabel={Errors to reference solution},
    xlabel={$h$},
    legend to name=legend:conv_h_t,  
    legend columns=3,
    every axis legend/.append style={text=black}, 
    legend style={font=\small, draw=black, fill=white, cells={anchor=west}, row sep=1ex, column sep=1em},
]
\addplot+[mark=o, thick, black] table[x=h,y={uLiL2quotlong}] {optimal-h_errors_bilaplace-Gamma.txt};
\addlegendentry{$\|e_\vu\|_{L^\infty(L^2)}$}

\addplot+[mark=square, thick, blue] table[x=h,y={xiLiL2quotlong}] {optimal-h_errors_bilaplace-Gamma.txt};
\addlegendentry{$\|e_{\xi}\|_{L^\infty(L^2)}$}

\addplot+[mark=x, thick, green] table[x=h,y={etaLiL2quotlong}] {optimal-h_errors_bilaplace-Gamma.txt};
\addlegendentry{$\|e_{\eta}\|_{L^\infty(L^2)}$}

\addplot+[mark=asterisk, thick, orange] table[x=h,y={gradetaLiL2quotlong}] {optimal-h_errors_bilaplace-Gamma.txt};
\addlegendentry{$\|\nabla e_{\eta}\|_{L^\infty(L^2)}$}

\addplot+[mark=+, thick, red] table[x=h,y={LapetaLiL2quotlong}] {optimal-h_errors_bilaplace-Gamma.txt};
\addlegendentry{$\|e_{\zeta}\|_{L^\infty(L^2)}$}

\addplot+[mark=otimes, thick, cyan] table[x=h,y={graduL2L2quotlong}] {optimal-h_errors_bilaplace-Gamma.txt};
\addlegendentry{$\|\nabla e_\vu\|_{L^2(L^2)}$}

\Convergence{0.9e-1}{1.5e-2}{2};
\Convergence{0.9e-1}{1.5e-2}{1};
\end{loglogaxis}
\end{tikzpicture}
\end{minipage}%
\hspace{0.01\textwidth}%
\begin{minipage}{0.49\textwidth}
\begin{tikzpicture}[scale=0.99]
\begin{loglogaxis}[
    width=\textwidth,
    ylabel={Errors to reference solution},
    xlabel={$\TS$},
]
\addplot+[mark=o, thick, black] table[x=dt,y={timeuLiL2quotlong}] {optimal-t_errors_bilaplace-Gamma.txt};
\addplot+[mark=square, thick, blue] table[x=dt,y={timexiLiL2quotlong}] {optimal-t_errors_bilaplace-Gamma.txt};
\addplot+[mark=x, thick, green] table[x=dt,y={timeetaLiL2quotlong}] {optimal-t_errors_bilaplace-Gamma.txt};
\addplot+[mark=asterisk, thick, orange] table[x=dt,y={timegradetaLiL2quotlong}] {optimal-t_errors_bilaplace-Gamma.txt};
\addplot+[mark=+, thick, red] table[x=dt,y={timeLapetaLiL2quotlong}] {optimal-t_errors_bilaplace-Gamma.txt};
\addplot+[mark=otimes, thick, cyan] table[x=dt,y={timegraduL2L2quotlong}] {optimal-t_errors_bilaplace-Gamma.txt};

\Convergence{1.4e-3}{12e-2}{1};
\end{loglogaxis}
\end{tikzpicture}
\end{minipage}\\[5pt]
\centerline{\ref{legend:conv_h_t}}

\caption{Mesh (left) and timestep (right) convergence for $\|e_\vu\|_{L^\infty(L^2)}$, $\|e_{\xi}\|_{L^\infty(L^2)}$, $\|e_{\eta}\|_{L^\infty(L^2)}$, $\|\nabla e_{\eta}\|_{L^\infty(L^2)}$, $\|e_{\zeta}\|_{L^\infty(L^2)}$ and $\|\nabla e_\vu\|_{L^2(L^2)}$. For a better comparison, the plots of the errors are shifted to start from the same point.}
\label{fig:conv_h_t}
\end{figure}
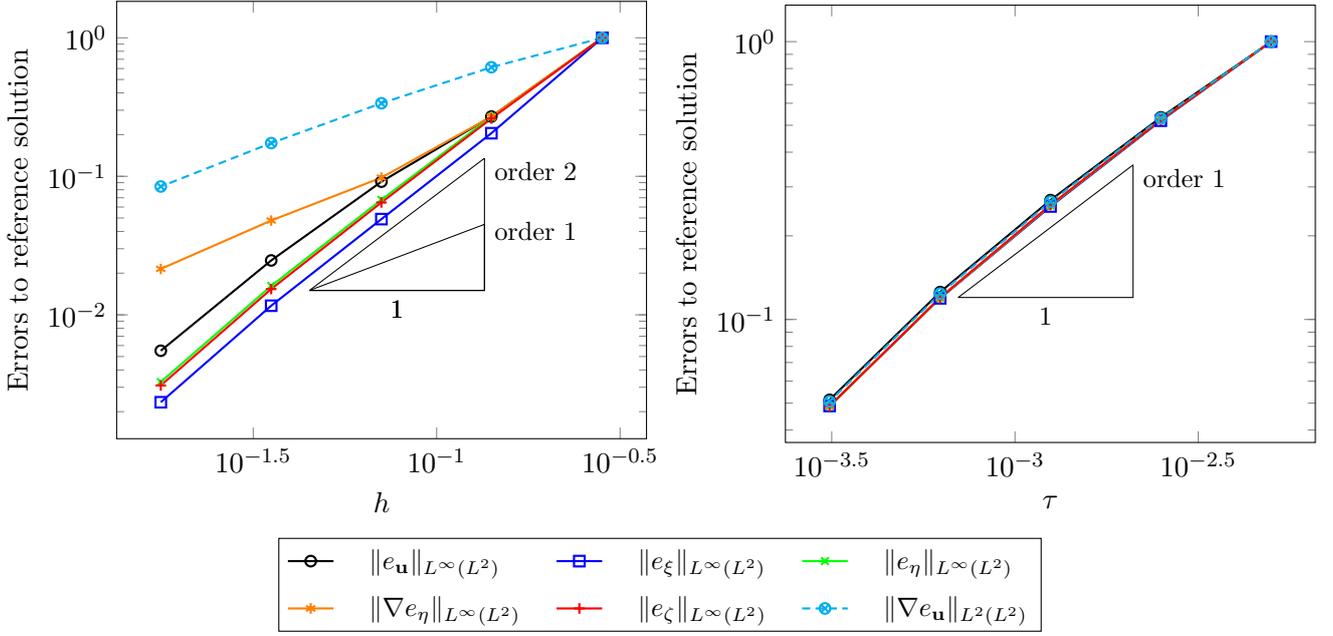

\subsection{Experiment 3: Comparison with a monolithic scheme from \cite{SST}}

The numerical scheme \eqref{SKM_ref} exhibits linear convergence in time and quadratic convergence in velocity. It is consistent with the rates observed for the monolithic Scheme-R \cite{SST}. The main difference lies in the algorithmic structure:
Scheme-R has two steps. Step 1 solves monolithically for the fluid and solid, finding $(\mathbf{u},\zeta,p)$ in $\Oref$, and Step 2 extends the displacement $\eta$ to the whole domain by solving a Laplace equation.  

In contrast, the current scheme \eqref{SKM_ref} does not solve the fluid-structure interaction monolithically. Instead, it splits the problem: Step 1 solves for the fluid $(\mathbf{u},p)$ in $\Oref$, Step 2 solves for the solid $(\zeta,\xi)$ on $\widehat{\Gamma}_S$, and Step 3 extends the displacement $\eta$ to the whole domain.  

On the finest mesh, Step~1 of Scheme-R involves 410\,880 degrees of freedom (DOFs), while the splitting scheme~\eqref{SKM_ref} uses 359\,360 DOFs in Step~1 and 640 DOFs in Step~2, thus formally reducing the size of the subproblems to be solved. 

On the same finest mesh, we compare the errors $\|\nabla e_{\bf u}\|_{L^2(L^2)}$ and $\|\nabla e_{\eta}\|_{L^\infty(L^2)}$ for several time steps $\TS$. As in the previous subsection, the errors are measured against the reference solution, here obtained by the monolithic Scheme-R on the finest mesh $h_{\rm min}=8.84\times 10^{-3}$ with the smallest time step $\TS=10^{-4}$. The convergence in time is shown in Figure~\ref{comparison_monolithic_split}, where both schemes exhibit the same rate, with the splitting scheme differing only by a constant. The monolithic scheme was re-implemented in Firedrake to ensure a fair comparison.

Although the splitting formulation reduces the number of unknowns in each subproblem, the total CPU time in this 2D-1D configuration is not lower than that of the monolithic scheme. The additional interpolation and projection steps between $\Oref$ and $\Sigma$ introduce extra computational overhead, which offsets the advantage of solving smaller systems. In higher-dimensional settings (e.g., a three-dimensional fluid coupled with a two-dimensional structure), however, the splitting approach is expected to yield a more significant computational benefit.

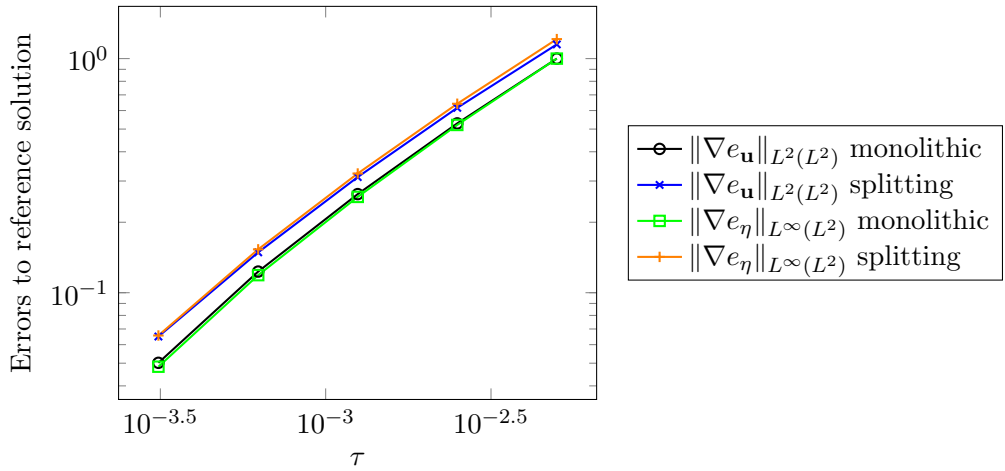
\begin{figure}[!htbp]
\begin{center}
\begin{tikzpicture}[scale=0.98]
\begin{loglogaxis}[
    width=8.0cm,
    ylabel={Errors to reference solution},
    xlabel={$\TS$},
    legend style={
    at={(1.85,0.5)},
    anchor=east,
    cells={anchor=west}},
    ]
\addplot+[mark=o, thick, black] table[x=dt,y={timegraduL2L2quotcomp}] {monolithic-t_errors_bilaplace.txt};
    \addlegendentry{$\|\nabla e_\vu\|_{L^2(L^2)}$ monolithic}
\addplot+[mark=x, thick, blue] table[x=dt,y={timegraduL2L2quotcomp}] {optimal-t_errors_bilaplace_compare.txt};
    \addlegendentry{$\|\nabla e_\vu\|_{L^2(L^2)}$ splitting}
\addplot+[mark=square, thick, green] table[x=dt,y={timegradetaLiL2quotcomp}] {monolithic-t_errors_bilaplace.txt};
    \addlegendentry{$\|\nabla e_{\eta}\|_{L^\infty(L^2)}$ monolithic}
\addplot+[mark=+, thick, orange] table[x=dt,y={timegradetaLiL2quotcomp}] {optimal-t_errors_bilaplace_compare.txt};
    \addlegendentry{$\|\nabla e_{\eta}\|_{L^\infty(L^2)}$ splitting}

\end{loglogaxis}
\end{tikzpicture}
\caption{Timestep convergence comparison monolithic vs. splitting scheme for $\|\nabla e_\vu\|_{L^2(L^2)}$ and $\|\nabla e_{\eta}\|_{L^\infty(L^2)}$.}\label{comparison_monolithic_split}
\end{center}
\end{figure}

\section{Conclusion}
In this work, we proposed and analyzed a linear, partitioned finite element scheme for the interaction between an incompressible viscous fluid and a thin deformable structure. The method allows the fluid and structure subproblems to be solved independently, while maintaining implicit velocity coupling at the interface, which ensures stability at the discrete level. Our analysis demonstrates unconditional energy stability and establishes optimal convergence rates in both space and time, all without relying on the assumptions of infinitesimal structural deformations or neglecting nonlinear fluid convection. Beyond the theoretical results, numerical experiments confirmed the predicted convergence orders and illustrated the robustness and practical effectiveness of the proposed partitioned approach.



\section*{\centering Funding}
The work of B. She is supported by National Natural Science Foundation of China (Grant No.12201437). 
K. Tůma has been supported by the project No. 23-05207S financed by the Czech Science Foundation, Czech Republic (GAČR) and by Charles University Research Centre, Czech Republic program No. UNCE/24/SCI/005. 
The work of T. Tian is supported by National Natural Science Foundation of China (Grant No.12401508).

\bibliographystyle{siamplain}

\appendix


\section{Appendix: Useful equalities and estimates}
\subsection{Proof of the error equation \eqref{EEQ}}\label{app_ee1}
In this part, we show how to obtain the equation \eqref{EEQ} satisfied by the errors. 
First, for any $k=1,\dots,N_T$ we subtract the weak formulation \eqref{wf2_ref} from the numerical scheme \eqref{equstability} and get 
\begin{equation}\label{TIK}
  \sum_{i=1}^7 T_i^k =0,  
\end{equation} 
where $T_i^k$ reads (keeping in mind that $R_i^k$, $i=1,\dots, 7$, are given in \eqref{RS})
\begin{align*}
T_1^k  = & \vrf\intOref{ (\eta_h^{k} \PDt \hvu_h^k   -  \eta^k \pdt \hvu^k  ) \cdot \hbfphi  } 
\\ = &
\vrf \intOref{ \big( \eta_h^k \PDt (\hvu_h^k - \hvu^k) + \eta_h^k (\PDt \hvu^k -\pdt\hvu^k ) + (\eta_h^k -\eta^k) \pdt \hvu^k \big) \cdot \hbfphi } 
\\= &
 \vrf \intOref{  \eta_h^k \PDt \eu ^k  \cdot \hbfphi }  + R^k_1,
\\
T_2^k  =&  \frac12 \vrf \intOref{ ( \PDt \eta_h^k \hvu_h^{k*}   - \pdt \eta^k  \hvu^k ) \cdot \hbfphi  } 
	\\=&
\frac12\vrf \intOref{ \Big( \PDt \eta_h^k (\hvu_h^{k*} - \hvu^{k*}) + ( \PDt \eta_h^k -   \pdt \eta^k )\hvu^{k*}   + \pdt \eta^k (\hvu^{k*} -\hvu^k ) \Big) \cdot \hbfphi}
\\ =&
\frac12\vrf \intOref{ \PDt \eta_h^k \eu ^{k*}  \cdot \hbfphi } 
+R_2^k,
\\
 T_3^k = &
\frac12 \vrf \intOref{  \big( \hbfphi \cdot (\Grad\hvu_h^k)      -  \hvu_h^k \cdot (\Grad\hbfphi)   \big)  \cdot  (\Jacob_h^k)^{-1}  \cdot \hvv_h^{k-1} \eta_h^k } 
\\&  - \frac12 \vrf \intOref{  \big( \hbfphi \cdot (\Grad\hvu^k)      -  \hvu^k \cdot (\Grad\hbfphi)   \big)  \cdot  (\Jacob^k)^{-1}  \cdot \hvv^{k} \eta^k } 
\\= &
\frac12 \vrf \intOref{ \Big( \hbfphi \cdot   (\Grad \euk ) 
-\euk   \cdot   (\Grad \hbfphi) \Big)
 \cdot (\Jacob_h^k)^{-1} \hvv_h^{k-1} \eta_h^k }
\\&  
+ \frac12 \vrf \intOref{ \Big( \hbfphi \cdot (\Grad \hvu^k) - \hvu^k \cdot (\Grad\hbfphi) \Big) \cdot \left(  (\Jacob_h^k)^{-1} \hvv_h^{k-1} \eta_h^k  -  (\Jacob^k)^{-1} \hvv^{k} \eta^k  \right)  } 
 =  R^k_3,
\\
 T_4^k = & 
 \intOrefB{ \hp_h^k \Grad \hbfphi : \Mhk - \hp^k \Grad \hbfphi : \Mk}
\\ = &  
\intOref{  e_p^k \Grad \hbfphi : \Mhk } 
+ \intOref{ \hp^k \Grad \hbfphi : \big(\Mhk -  \Mk\big)} 
=R_4^k,
\\
T_5^k = &
2 \mu \intOrefB{  \big( \Grad \hvu_h^k (\Jacob_h^k)^{-1}\big)^\rmS  :  \big(\Grad \hbfphi  (\Jacob_h^k)^{-1}\big)\eta_h^k   
- \big( \Grad \hvu^k (\Jacob^k)^{-1}\big)^\rmS  :  \big(\Grad \hbfphi (\Jacob^k)^{-1}\big) \eta^k    }
\\ = & 2 \mu \intOref{ \big( \Grad \euk  (\Jacob_h^k)^{-1} \big)^\rmS: \big(\Grad \hbfphi (\Jacob_h^k)^{-1} \big)\eta_h^k  } 
\\ & + 2 \mu \intOrefB{  \big( \Grad \hvu^k (\Jacob_h^k)^{-1} \big)^\rmS: (\Grad \hbfphi (\Jacob_h^k)^{-1}) \eta_h^k 
-\big( \Grad \hvu^k (\Jacob^k)^{-1} \big)^\rmS: (\Grad \hbfphi (\Jacob^k)^{-1})\eta^k}
\\ = & 2 \mu \intOref{ \big( \Grad \euk  (\Jacob_h^k)^{-1} \big)^\rmS: \big(\Grad \hbfphi (\Jacob_h^k)^{-1} \big)\eta_h^k  }  + R_5^k,
\\
  T_6^k =&  \vrs  \intS{( \PDt \xi_h^k  - \pdt \xi^k  )\psi} 
 =  \vrs\intS{\PDt e_\xi^k \psi} +R_6^k,
\\
  T_7^k =& a_s(\eta_h^{k+1},  \zeta_h^{k}, \psi)  -
  a_s(\eta^k,  \zeta^{k} ,\psi)  
 \\ = &
a_s(e_\eta^{k+1},e_\zeta^{k+1},\psi)
+ \gamma_1 \intS{\pdx (\eta^{k+1} -\eta^k) \pdx \psi}
+ \gamma_2 \intS{\pdx (\zeta^{k+1} -\zeta^k) \pdx \psi} 
 \\= &
a_s(e_\eta^{k+1},e_\zeta^{k+1},\psi)
+R_7^k.
 \end{align*}
Consequently, substituting the above expansions of the $T_i$-terms into \eqref{TIK} and shifting the $R_i$-terms to the right-hand-side, we derive \eqref{EEQ}.

\subsection{Preliminary estimates}
In this part we show some preliminary estimates and equalities. 
First, we show the estimates related to the time discretization operator $\PDt$ given in scheme \eqref{SKM_ref}. The technical details of proofs can be found in \cite[Appendix B.2]{SST}.
\begin{Lemma}
\label{lm_edt}
Let $\phi \in L^2((0,T)\times D)$ for $D \in \{ \Sigma, \Oref\}$. Then we have
\begin{subequations}
\begin{equation}\label{edts1}
 \TS \sumN \norm{ \PDt \phi ^{k} -  \pdt \phi^k}_{L^2(D)}^2 \aleq \TS^2  \norm{\pdtt \phi}_{L^2((0,T)\times D)}^2,
\end{equation}
\begin{equation}\label{edts2}
 \TS \sumN \norm{ \PDt \phi ^{k+1} -  \pdt \phi^k}_{L^2(D)}^2 \aleq \TS^2  \norm{\pdtt \phi}_{L^2((0,T)\times D)}^2.
\end{equation}
\end{subequations}
\end{Lemma}

\begin{Lemma}
Let $\eta \in W^{2,2}(\Sigma)$, $\xi=\pdt \eta$, $\zeta = - \Lapx \eta$, $\eta_h \in \Vsh$, $\xi_h^k=\PDt \eta_h^{k+1}$, $k=1,\ldots,N_T$,  $\zeta_h = -\Laph \eta_h$, and $\psi\in \Vsh$. Let $a_s$ be given by \eqref{as2} and the notation of the errors be given by \eqref{ers}. Then
\begin{subequations}
\begin{equation}\label{IM1}
 \delta_\xi^k = \PDt \delta_\eta^{k+1} + \Riesz (\PDt \eta^{k+1} -\pdt \eta^k),
\end{equation}
\begin{equation}\label{IM12}
     \delta_\zeta 
     = - \Laph \delta_\eta ,
\quad
 \intS{  \psi  \delta_\zeta  } 
= - \intS{ \Laph \psi  \;  \delta_\eta   } .
\end{equation}

\begin{equation}\label{IM2}
 \intS{\pdx I_\eta \pdx \psi}=0,\quad   \intS{\pdx I_\xi \pdx \psi}=0, 
\end{equation}
\begin{equation}\label{IM23}
\intS{I_\zeta \psi} =0, \quad 
\intS{\pdx I_\zeta \pdx\psi}=0,
\end{equation}
\begin{equation}\label{IM3}
   \intS{\pdx \delta_\zeta \pdx \psi}
  = \intS{\Laph \delta_\eta \Laph \psi},
\end{equation}
\begin{equation}\label{IM4}
\begin{aligned}
	\intS{\pdx \delta_\eta^{k+1} \pdx \delta_\xi^k}
	=& \intSB{\PDt \frac{|\pdx \delta_\eta^{k+1}|^2}{2}  + \frac{\TS}{2}|\PDt \pdx  	\delta_\eta^{k+1}|^2}
\\&+\intS{\pdx \delta_\eta^{k+1} \pdx (\PDt \eta^{k+1} - \pdt\eta^k )},
\end{aligned}
\end{equation}
\begin{equation}\label{IM5}
\begin{aligned}
	\intS{\pdx \delta_\zeta^{k+1} \pdx \delta_\xi^k}
	=& \intSB{\PDt \frac{| \delta_\zeta^{k+1}|^2}{2}  + \frac{\TS}{2}|\PDt   	\delta_\zeta^{k+1}|^2}
\\&- \intS{ \delta_\zeta^{k+1} \Lapx (\PDt \eta^{k+1} - \pdt\eta^k )},
\end{aligned}
\end{equation}
\begin{equation}\label{IM6}
\begin{aligned}
&a_s(e_\eta^{k+1},e_\zeta^{k+1},\delta_\xi^k )
=
 \PDt \intSB{ \frac{\gamma_1}{2} \abs{\pdx \delta_\eta^{k+1}}^2 + \frac{\gamma_2}{2} \abs{\delta_\zeta^{k+1}}^2 }
\\&\quad + \frac{\TS}2 \intSB{\gamma_1 |\PDt \pdx \delta_\eta^{k+1}|^2 +\gamma_2 |\PDt \delta_\zeta^{k+1}|^2  }
\\& \quad 
+\intSB{\gamma_1 \pdx \delta_\eta^{k+1} \pdx (\PDt \eta^{k+1} - \pdt\eta^k )
-\gamma_2 \delta_\zeta^{k+1} \Lapx (\PDt \eta^{k+1} - \pdt\eta^k )}.
\end{aligned}
\end{equation}
\end{subequations}
\end{Lemma}

\subsection{Secondary estimates} 

\begin{Lemma}\label{b5}
Let $G_{sL}$ be given by \eqref{REFB5}, then the following estimates of $G_{sL}$ hold:
\begin{equation}\label{estimate}
\begin{aligned}
G_{sL}\gtrsim -\TS^2 +\frac{\tau^2\gamma_1^2}{2\vrs}\norm{\delta_\zeta^{m+1}}^2_{L^2(\Sigma)}+\frac{\tau^2\gamma_2^2}{2\vrs}\norm{\Laph\delta_\zeta^{m+1}}^2_{L^2(\Sigma)}-\TS \summ \delta_E^k\\
-\summ\frac{\tau^3\gamma_1^2}{2\vrs}\norm{\delta_\zeta^{k+1}}^2_{L^2(\Sigma)}-\summ\frac{\tau^3\gamma_2^2}{2\vrs}\norm{\Laph\delta_\zeta^{k+1}}^2_{L^2(\Sigma)}. 
\end{aligned}
\end{equation}
\end{Lemma}

\begin{proof}Recalling  $ \Delta_\xi^k
=D_t (\gamma_1  \zeta_h^{k+1}  - \gamma_2 \Laph \zeta_h^{k+1} )$ from  \eqref{deltaxi} 
we have
\begin{equation*}
\TS \Delta_\xi^k= \gamma_1 (\delta_\zeta^{k+1}-\delta_\zeta^{k})+\gamma_1{\Riesz}(\zeta^{k+1}-\zeta^{k})-\gamma_2 \Laph  (\delta_\zeta^{k+1}-\delta_\zeta^{k})-\gamma_2 \Laph{\Riesz}(\zeta^{k+1}-\zeta^{k}).
\end{equation*}
Then, we reformulate $G_{sL}$ as 
 \begin{equation*}
G_{sL}=\TS\summ(Q_{1k}+Q_{2k}+Q_{3k})
\end{equation*}
with 
\begin{align*}
Q_{1k}=&
\tau\intS{\PDt \delta_\xi^k(\gamma_1(\delta_\zeta^{k+1}-\delta_\zeta^{k})-\gamma_2\Laph(\delta_\zeta^{k+1}-\delta_\zeta^{k}))}
\\
&+\frac{\tau}{\vrs}\intS{(\gamma_2\Laph\delta_\zeta^{k+1}-\gamma_1\delta_\zeta^{k+1})(\gamma_2\Laph(\delta_\zeta^{k+1}-\delta_\zeta^{k})-\gamma_1(\delta_\zeta^{k+1}-\delta_\zeta^{k}))},
\\ 
Q_{2k}=&-\intS{(\delta_\xi^{k}-\delta_\xi^{k-1})(\gamma_2\Laph {\Riesz}(\zeta^{k+1}-\zeta^{k})-\gamma_1 {\Riesz}(\zeta^{k+1}-\zeta^{k}))},
\\
Q_{3k}=&\frac{\tau}{\vrs}\intS{(\gamma_1\partial_{x_1}\delta_\eta^{k+1}+\gamma_2\partial_{x_1}\delta_\zeta^{k+1})(\gamma_2\Laph {\Riesz}(\zeta^{k+1}-\zeta^{k})-\gamma_1   \Riesz (\zeta^{k+1}-\zeta^{k}))},
\end{align*}

Using \eqref{IM1}, we obtain
\begin{equation*}
\begin{aligned}
Q_{1k}&
=\frac{\tau\gamma_1^2}{2\vrs}\left(\norm{\delta_\zeta^{k+1}}^2_{L^2(\Sigma)}-\norm{\delta_\zeta^{k}}^2_{L^2(\Sigma)}+\norm{\delta_\zeta^{k+1}-\delta_\zeta^{k}}^2_{L^2(\Sigma)}\right)\\
&+\frac{\tau\gamma_1\gamma_2}{\vrs}\left(\norm{\partial_{x_1}\delta_\zeta^{k+1}}^2_{L^2(\Sigma)}-\norm{\partial_{x_1}\delta_\zeta^{k}}^2_{L^2(\Sigma)}
+\norm{\partial_{x_1}(\delta_\zeta^{k+1}-\delta_\zeta^{k})}^2_{L^2(\Sigma)}\right)
\\&
+\frac{\tau\gamma_2^2}{2\vrs}\left(\norm{\Laph\delta_\zeta^{k+1}}^2_{L^2(\Sigma)}-\norm{\Laph\delta_\zeta^{k}}^2_{L^2(\Sigma)}+\norm{\Laph(\delta_\zeta^{k+1}-\delta_\zeta^{k})}^2_{L^2(\Sigma)}\right)\\&
+\frac{\gamma_1\tau}{2}\left(\norm{\partial_{x_1}\delta_\xi^k}^2_{L^2(\Sigma)}-\norm{\partial_{x_1}\delta_\xi^{k-1}}^2_{L^2(\Sigma)}+\norm{\partial_{x_1}(\delta_\xi^{k}-\delta_\xi^{k-1})}^2_{L^2(\Sigma)}\right)\\&+\underbrace{\tau\gamma_1\intS{ \partial_{x_1}(\delta_\xi^{k}-\delta_\xi^{k-1})\partial_{x_1}\Riesz(\xi^k-\PDt\eta^{k+1})}}_{T_3}\\
&+\frac{\gamma_2\tau}{2}\left(\norm{\Laph\delta_\xi^{k+1}}^2_{L^2(\Sigma)}-\norm{\Laph\delta_\xi^{k}}^2_{L^2(\Sigma)}+\norm{\Laph(\delta_\xi^{k}-\delta_\xi^{k-1})}^2_{L^2(\Sigma)}\right)\\&+\underbrace{\tau\gamma_2\intS{\Laph(\delta_\xi^{k}-\delta_\xi^{k-1})\Laph\Riesz(\xi^k-\PDt\eta^{k+1})}}_{T_4}.
\end{aligned}
\end{equation*}

By Young's inequality, \eqref{RieszE} and \eqref{edts2}, there holds
\begin{equation*}
T_3\gtrsim-\frac{\tau\gamma_1}{4}\norm{\partial_{x_1}(\delta_\xi^{k}-\delta_\xi^{k-1})}^2_{L^2(\Sigma)}-\tau^2\gamma_1\norm{\partial_t\partial_{x_1}\xi}^2_{L^2((t_k,t_{k+1})\times\Sigma)}.
\end{equation*}
Similarly,
\begin{equation*}
T_4\gtrsim-\frac{\tau\gamma_2}{4}\norm{\Laph(\delta_\xi^{k}-\delta_\xi^{k-1})}^2_{L^2(\Sigma)}-\tau^2\gamma_2\norm{\partial_t\partial_{x_1}^2\xi}^2_{L^2((t_k,t_{k+1})\times\Sigma)}.
\end{equation*}
Using Young's inequality, H\"older's inequality, \eqref{edts2} and \eqref{norm}, 
\begin{equation}\label{norm}
\norm{\Laph\eta}_{L^2(\Sigma)}\leq \norm{\partial_{x_1}^2\eta}_{L^2(\Sigma)}.
\end{equation}
we can get
\begin{equation*}
\begin{aligned}
\abs{ Q_{2k} }
&\leq\norm{\delta_\xi^{k}-\delta_\xi^{k-1}}_{L^2(\Sigma)}\norm{\gamma_1 \Riesz(\zeta^{k+1}-\zeta^{k})}_{L^2(\Sigma)}+\norm{\delta_\xi^{k}-\delta_\xi^{k-1}}_{L^2(\Sigma)}\norm{\gamma_2\Laph{\Riesz}(\zeta^{k+1}-\zeta^{k}))}_{L^2(\Sigma)}\\
&\leq\frac{\vrs}{4}\left(\norm{\delta_\xi^{k}}^2_{L^2(\Sigma)}+\norm{\delta_\xi^{k-1}}^2_{L^2(\Sigma)}\right) 
+\frac{2\gamma_1^2}{\vrs}\norm{\partial_{x_1}^2(\eta^{k+1}-\eta^k)}^2_{L^2(\Sigma)}+\frac{2\gamma_2^2}{\vrs}\norm{\Lapx(\zeta^{k+1}-\zeta^k)}^2_{L^2(\Sigma)}
\\&
\lesssim\frac{\vrs}{4}\left(\norm{\delta_\xi^{k}}^2_{L^2(\Sigma)}+\norm{\delta_\xi^{k-1}}^2_{L^2(\Sigma)}\right) +\frac{2\gamma_1^2\tau}{\vrs}\norm{\partial_{x_1}^2\xi}^2_{L^2((t_k,t_{k+1})\times\Sigma)}
+\frac{2\gamma_2^2\tau}{\vrs}\norm{\partial_{x_1}^4\xi}^2_{L^2((t_k,t_{k+1})\times\Sigma)}.
\end{aligned}
\end{equation*}
Similarly,
\begin{equation*}
\begin{aligned}
\abs{ Q_{3k} }
&\leq \frac{\tau^2}{2\vrs}\Big(\gamma_1^2\norm{\delta_\zeta^{k+1}}^2_{L^2(\Sigma)}+\gamma_2^2\norm{\Laph\delta_\zeta^{k+1}}^2_{L^2(\Sigma)}\Big)\\
&+\frac{1}{2\vrs}\Big(\gamma_1^2\norm{{\pdx \Riesz}(\eta^{k+1}-\eta^{k})}^2_{L^2(\Sigma)}+\gamma_2^2\norm{ {\pdx \Riesz}(\zeta^{k+1}-\zeta^{k})}^2_{L^2(\Sigma)}\Big)
\\&\lesssim \frac{\tau^2}{2\vrs}\left(\gamma_1^2\norm{\delta_\zeta^{k+1}}^2_{L^2(\Sigma)}+\gamma_2^2\norm{\Laph\delta_\zeta^{k+1}}^2_{L^2(\Sigma)}\right)\\
&+\frac{\gamma_1^2\tau}{2\vrs}\norm{\partial_{x_1}^2\xi}^2_{L^2((t_k,t_{k+1})\times\Sigma)}
+\frac{\gamma_2^2\tau}{2\vrs}\norm{\partial_{x_1}^4\xi}^2_{L^2((t_k,t_{k+1})\times\Sigma)}.
\end{aligned}
\end{equation*}

Hence, by collecting the above estimates, we get
\begin{equation*}
\begin{aligned}
G_{sL}&\gtrsim \frac{\tau^2\gamma_1^2}{2\vrs}\norm{\delta_\zeta^{m+1}}^2_{L^2(\Sigma)}+\summ\frac{\tau^2\gamma_1^2}{2\vrs}\norm{\delta_\zeta^{k+1}-\delta_\zeta^{k}}^2_{L^2(\Sigma)}+
\frac{\tau^2\gamma_1\gamma_2}{\vrs}\norm{\partial_{x_1}\delta_\zeta^{m+1}}^2_{L^2(\Sigma)}
\\&+\summ\frac{\tau^2\gamma_1\gamma_2}{\vrs}\norm{\partial_{x_1}(\delta_\zeta^{k+1}-\delta_\zeta^{k})}^2_{L^2(\Sigma)}
+\frac{\tau^2\gamma_2^2}{2\vrs}\norm{\Laph\delta_\zeta^{m+1}}^2_{L^2(\Sigma)}+\summ\frac{\tau^2\gamma_2^2}{2\vrs}\norm{\Laph(\delta_\zeta^{k+1}-\delta_\zeta^{k})}^2_{L^2(\Sigma)}
\\&
+\frac{\tau^2\gamma_1}{2}\norm{\partial_{x_1}\delta_\xi^{m+1}}^2_{L^2(\Sigma)}
+\summ\frac{\tau^2\gamma_1}{4}\norm{\partial_{x_1}(\delta_\xi^{k}-\delta_\xi^{k-1})}^2_{L^2(\Sigma)}+\frac{\tau^2\gamma_2}{2}\norm{\Laph\delta_\xi^{m+1}}^2_{L^2(\Sigma)}\\
&+\summ\frac{\tau^2\gamma_2}{4}\norm{\Laph(\delta_\xi^{k}-\delta_\xi^{k-1})}^2_{L^2(\Sigma)}
-\summ\frac{\tau\vrs}{2}\norm{\delta_\xi^{k}}^2_{L^2(\Sigma)}-\summ\frac{\tau^3\gamma_1^2}{2\vrs}\norm{\delta_\zeta^{k+1}}^2_{L^2(\Sigma)}
\\&-\summ\frac{\tau^3\gamma_2^2}{2\vrs}\norm{\Laph\delta_\zeta^{k+1}}^2_{L^2(\Sigma)}-\frac{5\gamma_1^2\tau^2}{2\vrs}\norm{\partial_{x_1}^2\xi}^2_{L^2((0,T)\times\Sigma)}
-\frac{5\gamma_2^2\tau^2}{2\vrs}\norm{\partial_{x_1}^4\xi}^2_{L^2((0,T)\times\Sigma)}\\
&-\tau^3\Big(\gamma_1\norm{\partial_t\partial_{x_1}\xi}^2_{L^2((0,T)\times\Sigma)}+\gamma_2\norm{\partial_t\partial_{x_1}^2\xi}^2_{L^2((0,T)\times\Sigma)}\Big)\\
&\gtrsim -\TS^2 +\frac{\tau^2\gamma_1^2}{2\vrs}\norm{\delta_\zeta^{m+1}}^2_{L^2(\Sigma)}+\frac{\tau^2\gamma_2^2}{2\vrs}\norm{\Laph\delta_\zeta^{m+1}}^2_{L^2(\Sigma)}-\TS \summ \delta_E^k\\
&-\summ\frac{\tau^3\gamma_1^2}{2\vrs}\norm{\delta_\zeta^{k+1}}^2_{L^2(\Sigma)}-\summ\frac{\tau^3\gamma_2^2}{2\vrs}\norm{\Laph\delta_\zeta^{k+1}}^2_{L^2(\Sigma)},
\end{aligned}
\end{equation*}
which proves \eqref{estimate}.
\end{proof}

\subsection{Proof of estimates \eqref{res}}\label{app_res}

\begin{proof}
Property \eqref{IM2} implies
\begin{equation*}
   \intS{\vrs \PDt I_\xi^k  \Delta_\xi^k}=0.
\end{equation*}
Consequently, the estimates for $G_f$, $G_s$, and $R_i$ ($i=1,\cdots,5$) coincide with those established in \cite[Appendix B.4]{SST}. We now proceed to analyze the remaining terms.
\paragraph{$R^k_6$-term} By Young's inequality and \eqref{edts1} we obtain
\begin{align*}
 \abs{\TS \summ R^k_6} & = \Abs{ \TS \summ \vrs\intS{ (\PDt \xi^k -\pdt \xi^k) (\delta_\xi^k +\frac{\TS^2}{\vrs} \Delta_\xi^k )} }
  \\&\aleq \frac{\TS^2}{4\vrs}  \norm{\pd_t^2 \xi }_{L^2((0,T)\times\Sigma)}^2 + \TS \summ \intS{\vrs \abs{\delta_\xi}^2}+\TS^3 \summ\intS{ \abs{ \TS \Delta_\xi^k}^2 } .
\end{align*}
\paragraph{$R^k_7$-term.} 

\begin{align*}
& \Abs{ \TS \summ R^k_7 } =  \Abs{ \TS \summ  \gamma_1 \intS{\pdx (\eta^{k+1} -\eta^k) \pdx  (\delta_\xi^k+ \frac{\TS^2}{\vrs}\Delta_\xi^k)} 
+  \TS \summ  \gamma_2 \intS{\pdx (\zeta^{k+1} -\zeta^k) \pdx (\delta_\xi^k+ \frac{\TS^2}{\vrs}\Delta_\xi^k)} }
\\
&\aleq \TS \summ \TS \Big(\norm{\pdx^2D_t\eta^k}_{L^2(\Sigma)}+\norm{\pdx^2D_t\zeta^k}_{L^2(\Sigma)}\Big)\norm{\delta_\xi^k+ \frac{\TS^2}{\vrs}\Delta_\xi^k}_{L^2(\Sigma)}
\\
&\aleq \TS \summ \norm{\delta_\xi^k}_{L^2(\Sigma)}^2 + \TS^3 \summ \norm{\TS \Delta_\xi^k}_{L^2(\Sigma)}^2+\TS^2(\norm{\pdx^4\xi}_{L^2((0,T)\times \Sigma)}^2+\norm{\pdx^2\xi}_{L^2((0,T)\times \Sigma)}^2).
\end{align*}
%

Note that 
\begin{multline*}
\intS{\abs{\TS \Delta_\xi^k}^2 }\lesssim \norm{\delta_\zeta^{k+1}-\delta_\zeta^k}_{L^2(\Sigma)}^2+\norm{\Riesz(\zeta^{k+1}-\zeta^k)}_{L^2(\Sigma)}^2+\norm{\Laph(\delta_\zeta^{k+1}-\delta_\zeta^k)}_{L^2(\Sigma)}^2 \\ +\norm{\Laph\Riesz(\zeta^{k+1}-\zeta^k)}_{L^2(\Sigma)}^2 
\lesssim \norm{\delta_\zeta^{k+1}-\delta_\zeta^k}_{L^2(\Sigma)}^2 +\norm{\Laph(\delta_\zeta^{k+1}-\delta_\zeta^k)}_{L^2(\Sigma)}^2 + \TS^2.
\end{multline*}
we have  
\begin{equation*}
  \TS^3 \summ \norm{\TS \Delta_\xi^k}_{L^2(\Sigma)}^2 \lesssim  
  \TS^3 \summ  \Big( \norm{\delta_\zeta^{k+1} }_{L^2(\Sigma)}^2+\norm{\Laph \delta_\zeta^{k+1}}_{L^2(\Sigma)}^2 \Big) +\TS^4.
\end{equation*}
Consequently, collecting all the above estimates we get 
\begin{multline*}
\Abs{ \TS \summ \sum_{i=1}^7 R^k_i + G_f + G_s} 
\aleq \TS^2 + h^2  
+ c \TS \summ \delta_E^k 
+ 2 \alpha  \mu   \TS \summ \intOref{\left| \Grad \delta_\vu^k (\Jacob_h^k)^{-1}\right|^2 \eta_h^k}
\\ +\TS^3 \summ  \Big( \norm{\delta_\zeta^{k+1} }_{L^2(\Sigma)}^2+\norm{\Laph \delta_\zeta^{k+1}}_{L^2(\Sigma)}^2 \Big).
\end{multline*}
which proves \eqref{res}.

\end{proof}

\end{document}